\documentclass[11pt,oneside,english]{amsart}

\usepackage[T1]{fontenc}
\usepackage[latin9]{inputenc}
\usepackage{geometry}
\usepackage{color}
\usepackage{verbatim}
\usepackage{mathrsfs}
\usepackage{amsthm}
\usepackage{amssymb}
\usepackage{bbm}
\usepackage{tikz}
\usepackage{pgfplots}
\usepackage{babel}
\usepackage{amsfonts}
\usepackage{amsmath}
\usepackage{hyperref}
\usepackage{mathtools}
\usepackage{dsfont}
\usepackage{subfig}
\usepackage{url}
 
\makeatletter

\numberwithin{equation}{section}

\newtheorem{theorem}{Theorem}[section]
\newtheorem{lemma}[theorem]{Lemma}
\newtheorem{proposition}[theorem]{Proposition}

\newtheorem{remark}{Remark}

\newtheorem{hypothesis}{Hypothesis}

\numberwithin{remark}{section}

\newcommand \N {\mathbb{N}} 
\newcommand \R {\mathbb{R}}

\newcommand \cP {\mathcal P}
\newcommand \cS {\mathcal S}
\newcommand \cD {\mathcal D}
\newcommand \cE {\mathcal E}
\newcommand \cQ {\mathcal Q}

\newcommand \bP {\mathbb P}

\newcommand \ee {\mathrm{e}}
\newcommand \dd {\mathrm{d}}

\newcommand \arctanh {\mathrm{arctanh\,}}

\newcommand{\abs}[1]{\left\vert{#1}\right\vert}
\newcommand{\vprod}[2]{\left\langle {#1}, {#2}\right\rangle}

\newcommand{\si}{\sigma}

\newcommand{\cAl}{\abs{A_{\ell,n}}}
\newcommand{\oln}{\omega_{\ell,n}}
\newcommand{\capa}{\hbox{\rm cap}}
 
\pgfplotsset{compat=1.16}

\begin{document}

\title[]{Metastability for Glauber dynamics on\\ 
the complete graph with coupling disorder}

\author{A.\ Bovier}
\address{Institut f\"ur Angewandte Mathematik, Rheinische Friedrich-Wilhelms-Universit\"at Bonn, 
Endenicher Allee 60, 53115 Bonn, Germany.}
\email{bovier@uni-bonn.de}

\author{F.\ den Hollander}
\address{Mathematical Institute, Leiden University,  P. O. Box 9512, 2300 RA Leiden, The Netherlands.}
\email{denholla@math.leidenuniv.nl}

\author{S.\ Marello}
\address{Institut f\"ur Angewandte Mathematik, Rheinische Friedrich-Wilhelms-Universit\"at Bonn, 
Endenicher Allee 60, 53115 Bonn, Germany.}
\email{marello@iam.uni-bonn.de}

\begin{abstract}
Consider the complete graph on $n$ vertices. To each vertex assign an Ising spin that can take the values $-1$ or $+1$. 
Each spin $i \in [n]=\{1,2,\dots, n\}$ interacts with a magnetic field $h \in [0,\infty)$, while each pair of spins $i,j \in [n]$ interact with each other at coupling strength $n^{-1} J(i)J(j)$, where $J=(J(i))_{i \in [n]}$ are i.i.d. non-negative random variables drawn from a probability distribution with finite support. Spins flip according to a Metropolis dynamics at inverse temperature $\beta \in (0,\infty)$. We show that there are critical thresholds $\beta_c$ and $h_c(\beta)$ such that, in the limit as $n\to\infty$, the system exhibits metastable behaviour if and only if $\beta \in (\beta_c, \infty)$ and $h \in [0,h_c(\beta))$. Our main result is a sharp asymptotics, up to a multiplicative error $1+o_n(1)$, of the average crossover time from any metastable state to the set of states with lower free energy. We use standard techniques of the potential-theoretic approach to metastability. The leading order term in the asymptotics does not depend on the realisation of $J$, while the correction terms do. The leading order of the correction term is $\sqrt{n}$ times a centred Gaussian random variable with a complicated variance depending on $\beta,h$, on the law of $J$ and on the metastable state. The critical thresholds $\beta_c$ and $h_c(\beta)$ depend on the law of $J$, and so does the number of metastable states. We derive an explicit formula for $\beta_c$ and identify some properties of $\beta \mapsto h_c(\beta)$. Interestingly, the latter is not necessarily monotone, meaning that the metastable crossover may be re-entrant. 
 \end{abstract}

\keywords{Curie-Weiss, Glauber dynamics, disorder, metastability.}
\subjclass[2020]{60K35; 60K37; 82B20; 82B44; 82C44}
\thanks{AB and SM were supported through DFG Excellence Cluster GZ 2047/1, Projekt-id 390685813, and DFG SFB 1060, Projektnummer 211504053. FdH was supported through NWO Gravitation Grant 024.002.003-NETWORKS, and by the Alexander von Humboldt Foundation during sabbatical leaves in Bonn and Erlangen in the Fall of 2019 and 2020.}

\date{\today}

\maketitle


\section{Introduction and main results}
\label{sec:intro}


\subsection{Background}
\label{sec:background}

Interacting particle systems evolving according to a Metropolis dynamics associated with an energy functional called the Hamiltonian, may  be trapped for a long time near a state that is a local minimum of the free energy, but not a global minimum. The deepest local minima are called \emph{metastable states}, the global minimum is called the \emph{stable state}. The transition from a metastable state to the stable state marks the relaxation of the system to equilibrium. To describe this relaxation, one needs to identify the set of critical configurations the system must attain in order to achieve this transition and to compute the crossover time. These critical configurations correspond to saddle points in the free energy landscape.

Metastability for interacting particle systems on \emph{lattices} has been studied intensively in the past. For a summary, we refer the reader to the monographs by Olivieri and Vares~\cite{OV05}, and Bovier and den Hollander~\cite{BdH15}. Successful attempts towards understanding metastable behaviour in random environments were made for the random field Curie-Weiss model, by Mathieu and Picco~\cite{MP98}, Bovier, Eckhoff, Gayrard and Klein~\cite{BEGK1} and Bianchi, Bovier and Ioffe~\cite{BBI09,BBI12}. Recently, there has been interest in metastability for interacting particle systems on \emph{random graphs}. This is challenging, because the crossover times typically depend on the realisation of the graph. In den Hollander and Jovanovksi \cite{dHO20} and Bovier, Marello and Pulvirenti \cite{BMP21}, Glauber dynamics on dense \emph{Erd\H{o}s-R\'enyi} random graphs was analysed. Earlier work on metastability for Glauber dynamics on sparse random graphs can be found in Dommers~\cite{D17} (random regular graph) and Dommers, den Hollander, Jovanovski and Nardi~\cite{DdHJN17} (configuration model). The present paper is a first step towards the study of metastability for Glauber dynamics on \emph{Chung-Lu}-like random graphs.

To the best of our knowledge, Tindemans and Capel~\cite{TC74} and Dommers, Giardin\`a, Giberti, van der Hofstad and Prioriello~\cite{DGGHP16} are the only references where the model with the interaction Hamiltonian in \eqref{eq:Hn} below has been studied in detail. Both focus on equilibrium properties only.


\subsection{Glauber dynamics on the complete graph with coupling disorder}
\label{sec:model}

Let $\mathcal{K}_n$ be the complete graph on $n$ vertices. Each vertex carries an Ising spin that can take the values $-1$ or $+1$. Let $S_n = \{-1,+1\}^{[n]}$ denote the set of spin configurations on $\mathcal{K}_n$, where $[n]=\{1,2,\dots, n\}$. Let $(\Omega,\mathcal{F},\cP)$ be an abstract probability space, and let $J=(J(i))_{i \in [n]}$ be a sequence of i.i.d.\ random variables on this probability space taking values in a \emph{finite} set $\{a_1,\dots,a_k\} \subset [0,\infty)$ of cardinality $k \in \N$. The distribution of these random variables is given by 
\begin{equation}
		\cP\left(J(i)=a_{\ell}\right)=\omega_{\ell} \in (0,1), \quad i \in [n], \ell \in [k],
\end{equation}
with $\sum_{\ell \in [k]}\omega_{\ell}=1$.

Let $H_n\colon \cS_n \to \R$ be the interaction Hamiltonian defined by
\begin{equation}
\label{eq:Hn}
H_n(\sigma) \equiv - \frac{1}{n} \sum_{\substack{i , j \in [n] \\ i< j}} J(i) J(j)\, \sigma(i)\sigma(j) 
- h \sum_{i \in [n]} \sigma(i),
\qquad \sigma \in \cS_n,
\end{equation}
where $h \in [0,\infty)$ is the magnetic field. We consider \emph{Glauber dynamics} on $\cS_n$, defined as the continuous-time Markov process with transition rates 
\begin{equation}
r_n(\sigma,\sigma') =
\begin{cases}
\ee^{-\beta [H_n(\sigma')-H_n(\sigma)]_{+}}, 
& \mbox{if } \sigma' \sim \sigma,\\
0, & \mbox{otherwise},
\end{cases}
\qquad \sigma,\sigma' \in \cS_n,
\label{eq:rate r}
\end{equation}
where $\beta \in (0,\infty)$ is the inverse temperature, $\sigma'\sim\sigma$ means that $\sigma'$ differs from $\sigma$ by a single spin-flip and $[\cdot]_+$ is the positive part. This dynamics is  \emph{reversible} with respect to  the \emph{Gibbs measure} 
\begin{equation}
\label{eq:mu}
\mu_n(\sigma) \equiv \frac{1}{Z_n}\,\ee^{-\beta H_n(\sigma)}, 
\qquad \sigma \in \cS_n,
\end{equation}
where the normalising constant $Z_n$ is called the partition sum. Note that the reference measure for \eqref{eq:mu} is the \emph{counting measure} on $\cS_n$. We write
\begin{equation}
\label{eq:ERdyn}
(\sigma_t) _{t \geq 0}, \qquad \sigma_t\in \cS_n,
\end{equation} 
to denote a path of the Glauber dynamics on $\cS_n$, and $\mathbb{P}_\sigma$ and $\mathbb{E}_\sigma$ to denote probability and expectation on path space given $\sigma_0=\sigma$ (we suppress $J,h, \beta$ and $n$ from the notation).  

For fixed $n$, if $h=0$ the Hamiltonian in \eqref{eq:Hn} has two global minima, at $\sigma \equiv +1$ and $\sigma \equiv -1$, while if $h>0$ it achieves a global minimum at $\sigma \equiv +1$ and a local minimum at $\sigma \equiv -1$. The latter is the deepest local minimum not equal to the global minimum (at least for $h$ small enough). However, in the limit as $n\to\infty$, these do \emph{not} form a metastable pair of configurations because \emph{entropy} comes into play. 


\subsection{Metastability on the complete graph with coupling disorder}
\label{sec:metatheorems}

In this section we state our main results.
 

\subsubsection{Empirical magnetisations}
 
The relevant quantity to monitor in order to characterise the metastable behaviour is the \emph{disorder weighted magnetisation} 
\begin{equation}
K_n(\sigma) = \frac{1}{n} \sum_{i \in [n]} J(i)\sigma(i), \qquad \sigma \in \cS_n.
\end{equation} 
The following quantities will be essential for \emph{coarse-graining}. Define the \emph{level sets} 
\begin{equation}
\label{eq:Asets}
A_{\ell,n} \equiv \{i \in [n]\colon\, J(i) = a_{\ell}\}, \qquad \ell \in [k],
\end{equation}
and the \emph{level magnetisations}
\begin{equation}
\label{eq:levmagdef}
m_{\ell,n}(\sigma) \equiv \frac{1}{\cAl} \sum_{i \in A_{\ell,n}} \sigma(i), \qquad \ell \in [k],\,\sigma \in \cS_n.
\end{equation}
Put
\begin{equation}
\label{eq:magvec}
m_n(\sigma) = (m_{\ell,n}(\sigma)\big)_{\ell \in [k]} \in [-1,1]^k, \qquad \sigma \in \cS_n, 
\end{equation}
and note that $K_n(\sigma) = \frac{1}{n}\sum_{\ell \in [k]} a_\ell\,\cAl\,m_{\ell,n}(\sigma)$ depends on $\si$ only through $m_{n}(\sigma)$. Thus, with abuse of notation, we may define 
\begin{equation}\label{eq:Kn}
K_n(m) \equiv \frac{1}{n}\sum_{\ell \in [k]} a_\ell \cAl m_{\ell}, \qquad m=(m_{\ell})_{\ell \in [k]}\in [-1,1]^k, 
\end{equation} 
so that $K_n(\sigma)=K_n(m_{n}(\sigma))$.


\subsubsection{Thermodynamic limit}

As $n\to\infty$, by the law of large numbers the random function $K_n$ converges uniformly in probability to a deterministic function $K$ given by 
\begin{equation}
\label{eq:KFirst}
K(m) =  \sum_{\ell \in [k]} a_{\ell}\,\omega_{\ell}\, m_{\ell}, \qquad m=(m_{\ell})_{\ell \in [k]} \in [-1,1]^k.
\end{equation}
Similarly, the random free energy function $F_n$ converges uniformly in probability to a deterministic function $F_{\beta,h}$ (see \eqref{eq:defFn} and \eqref{eq:F} below for explicit formulas). In Section~\ref{sec:profiles}, we show that the stationary points of $F_{\beta,h}$ are given by $\mathbf{m} = (\mathbf{m}_\ell)_{\ell\in[k]}$, where
\begin{equation}\label{eq:ml}
\mathbf{m}_\ell = \tanh(\beta[a_{\ell} K(\mathbf{m})+ h]), \qquad \ell \in [k].
\end{equation}
Note that, via \eqref{eq:ml}, the $k$-dimensional vector $\mathbf{m}$ is fully determined by the real number $K(\mathbf{m})$. Therefore, finding the stationary points of $F_{\beta,h}$ reduces to finding the solutions of the equation
\begin{equation}
\label{eq:K}
K = T_{\beta,h}(K), \qquad T_{\beta,h}(K) 
= \sum_{\ell \in [k]} a_{\ell}\,\omega_{\ell}\,\tanh(\beta[a_{\ell}K+h]).
\end{equation}


\subsubsection{Metastable regime}

It turns out that the critical inverse temperature $\beta_c$ is given by
\begin{equation} 
\label{eq:crt}
\beta_c = \left[\sum_{\ell \in [k]} a_{\ell}^2\omega_{\ell}\right]^{-1}.
\end{equation}
Namely, if $\beta \in (0, \beta_c]$, then the system is not in the metastable regime for any $h \in [0,\infty)$, while if $\beta \in (\beta_c,\infty)$, then, for $h \in [0,\infty)$ small enough, it is in the metastable regime (i.e., \eqref{eq:K} has more than one solution at which $T_{\beta,h}$ is not tangent to the diagonal). Given $\beta \in (\beta_c,\infty)$, the critical magnetic field $h_c(\beta)$ is the minimal value of $h$ for which the system is not metastable. The \emph{metastable regime} is thus
\begin{equation}
\label{eq:metregdis}
\beta \in (\beta_c,\infty), \qquad h \in [0,h_c(\beta)).
\end{equation} 
In Section~\ref{sec:profiles}, we  show that $\beta \mapsto h_c(\beta)$ is continuous on $(\beta_c,\infty)$, with
\begin{equation}
\lim_{\beta \downarrow \beta_c} h_c(\beta) = 0,
\qquad
\lim_{\beta \to \infty} h_c(\beta) = C \in (0,\infty),
\end{equation}
where the explicit value of $C$ is given in \eqref{eq:hcLim} below.
Interestingly, $\beta \mapsto h_c(\beta)$ is not necessarily monotone, i.e., the metastable crossover may be \emph{re-entrant}. 

It turns out that there exists an $\ell \in [k]$ (depending on $\beta,h$ and on the law of the components of $J$), such that $F_{\beta,h}$ has $2\ell+1$ stationary points.


\subsubsection{Metastable crossover}\label{sec:crossover}

Let $\mathcal{M}_n$ be the set of minima of $F_n$. Given $\mathbf{m} \in \mathcal{M}_n$, define 
\begin{equation}
\label{eq:defMn}
\mathcal{M}_n(\mathbf{m}) \equiv \{m \in \mathcal{M}_n\backslash\mathbf{m} \colon\, F_{n}(m) \leq F_{n}(\mathbf{m})\}.
\end{equation} 
Let $\mathcal{G}(A,B)$ be the gate between two disjoint subsets $A$ and $B$ of $\mathcal{M}_n$. We refer to \cite[Section 10.1]{BdH15} for a precise definition of the gate.

Fix $\mathbf{m}_n \in \mathcal{M}_n$ as the initial magnetisation. 
Throughout the paper we assume that the following hypotheses hold for $\mathbf{m}_n$.
\begin{hypothesis}
\label{hyp}
\noindent
\begin{enumerate}
\item 
$\mathcal{M}_n(\mathbf{m}_n)$ is non-empty.
\item 
\label{hyp:non-deg}
The Hessian of $F_n$ has only non-zero eigenvalues at $\mathbf{m}_n$ and at all the points in $\mathcal{G}(\mathbf{m}_n,\mathcal{M}_n(\mathbf{m}_n))$.
\item 
\label{hyp:unique}
There is a \emph{unique} point $\mathbf{t}_n$ in $\mathcal{G}(\mathbf{m}_n,\mathcal{M}_n(\mathbf{m}_n))$, which will often be called simply saddle point.
\item
\label{hyp:L} 
The saddle point $\mathbf{t}_n$ is such that $r_{\ell}[\cAl(1-\mathbf{t}_{\ell,n}^2)]^{-1}$ takes distinct values for different $\ell \in [k]$, where $r_{\ell}$ is defined in \eqref{eq:r_ell} below.
\end{enumerate}
\end{hypothesis}
\noindent

\noindent
Hypothesis~\ref{hyp}\eqref{hyp:non-deg} and \eqref{hyp:unique} are made to avoid complications. Hypothesis~\ref{hyp}\eqref{hyp:L} is needed in the proof of Lemma~\ref{lem:gamma} below (as in \cite[Lemma 14.9]{BdH15}). Neither is very restrictive: if for some parameter choice they fail, then after an infinitesimal parameter change they hold. Moreover, if Hypothesis~\ref{hyp}\eqref{hyp:unique} fails, it is sufficient to compute separately the contribution to the crossover time of the various saddle points in the gate.

Let $\cS_n[\mathbf{m}_n]$ and $\cS_n[\mathcal{M}_n(\mathbf{m}_n)]$ denote the sets of configurations in $\cS_n$ for which the level magnetisations are $\mathbf{m}_n$ and are contained in $\mathcal{M}_n(\mathbf{m}_n)$, respectively. For $A \subset \cS_n$, write
\begin{equation}
\label{eq:tau_def}
\tau_A = \{t \geq 0\colon\, \sigma_t \in A,\,\sigma_{t-} \notin A\}
\end{equation} 
to denote the first hitting time or return time of $A$.

We next state our main results for the crossover time. Theorem~\ref{thm:metER} provides a sharp asymptotics for the average crossover time from any metastable state to the set of states with lower free energy. Theorem~\ref{thm:exp_law} shows that asymptotically the crossover time is exponential on the scale of its mean, a property that is standard for metastable behaviour.   

\begin{theorem}[{\bf Average crossover time with coupling disorder}]
\label{thm:metER}
$\mbox{}$\\
Let $\mathbb{A}_n(\cdot)$ be the $k \times k$ Hessian matrix defined in \eqref{eq:defAnm} below, and $\gamma_{n}$ the unique negative solution of the equation in \eqref{eq:gamma} below. 
For every $\mathbf{m_n} \in \mathcal{M}_n$ satisfying Hypothesis~\ref{hyp} and within the metastable regime \eqref{eq:metregdis}, uniformly in $\sigma \in \cS_n[\mathbf{m_n}]$, and with $\cP$-probability tending to $1$,
\begin{equation}
\label{CWERasymp}
\begin{split}
\mathbb{E}_\sigma\left[\tau_{\cS_n[\mathcal{M}_n(\mathbf{m_n})]}\right] 
&=[1+o_n(1)]\,\sqrt{\frac{[-\det(\mathbb{A}_n(\mathbf{t}_n))]}{\det (\mathbb{A}_n(\mathbf{m}_n))}}
\left(\frac{\pi }{2 \beta(-\gamma_n)}\right) \ee^{\beta n[F_n(\mathbf{t}_n)-F_n(\mathbf{m}_n)]}.
\end{split}
\end{equation} 
\end{theorem}

\begin{theorem}[{\bf Exponential law with coupling disorder}] 
\label{thm:exp_law}
$\mbox{}$\\
For every $\mathbf{m_n} \in \mathcal{M}_n$ satisfying Hypothesis~\ref{hyp} and within the metastable regime \eqref{eq:metregdis}, uniformly in $\sigma \in \cS_n[\mathbf{m}_n]$ and with $\cP$-probability tending to $1$,
\begin{equation}
\mathbb{P}_{\si} \left( \tau_{\cS_n[\mathcal{M}_n(\mathbf{m}_n)]}> t \, 
\mathbb{E}_\sigma\left[\tau_{\cS_n[\mathcal{M}_n(\mathbf{m}_n)]}\right] \right)
= [1+o_n(1)]\,\ee^{-t}, \qquad t \geq 0.
\end{equation}
\end{theorem}

As the average crossover time estimated in Theorem~\ref{thm:metER} is a random variable, we next provide more information on the randomness of the quantity in the right-hand side of \eqref{CWERasymp}, which depends on the realisation of the random variable $J$. The prefactor in \eqref{CWERasymp} converges with $\cP$-probability tending to $1$ to a deterministic limit, which depends on the law of $J$ but not on the realisation of $J$. However, the exponent does not converge to a deterministic limit. In Theorem~\ref{thm:metER_lim} we compute the exponent up to order $O(1)$. Recall that $F_n \to F_{\beta,h}$, $\mathbf{m}_n\to \mathbf{m}$ and $\mathbf{t}_n\to \mathbf{t}$ as $n\to\infty$.

\begin{theorem}[{\bf Randomness of the exponent}]
\label{thm:metER_lim}
$\mbox{}$\\
For every $\mathbf{m_n} \in \mathcal{M}_n$ satisfying Hypothesis~\ref{hyp} and within the metastable regime \eqref{eq:metregdis}, in distribution,
\begin{equation}
\label{CWERasymp_lim}
n [F_n(\mathbf{t}_n)-F_n(\mathbf{m}_n)] = n[F_{\beta,h}(\mathbf{t})-F_{\beta,h}(\mathbf{m})] + Z\sqrt{n} + O(1),
\end{equation} 
where $Z$ is a normal random variable with mean zero and variance in $(0,\infty)$, defined on $(\Omega,\mathcal{F}, \cP)$ and independent of $J$.
\end{theorem}

\noindent
The variance of $Z$ turns out to be a complicated function of $\beta$, $h$ and the distribution of $J$. We refer to Section~\ref{sec:correction} for further details. Computing the exponent up to order $1$ is in principle possible, but the formulas become rather complicated. Without this precision the prefactor in \eqref{CWERasymp} is asymptotically negligible. Still, knowing this prefactor allows us to determine what the leading order behaviour of the randomness is.

\subsection{Discussion on the continuous case}
Bianchi, Bovier and Ioffe \cite{BBI09,BBI12} study the Curie-Weiss model with a \emph{random magnetic field} whose distribution is continuous. Lumping techniques work for discrete distributions but not for continuous distributions. The latter require coarse-graining techniques to approximate the continuous distribution by a sequence of discrete distributions. In the present paper we consider pair interaction random variables with a discrete distribution only. It seems hard to obtain results with a similar precision for continuous distributions. The techniques employed in \cite{BBI09,BBI12} do not carry over, because the error introduced by the coarse-graining turns out to be quadratic rather than linear.

\subsection{Techniques and outline}
\label{sec:disc}

In order to prove Theorems~\ref{thm:metER}--\ref{thm:metER_lim} we use the potential-theoretic approach to metastability developed in Bovier, Eckhoff, Gayrard and Klein~\cite{BEGK2,BEGK3}. More specifically, we first find a sharp approximation of the Dirichlet form associated with the coarse-grained dynamics. We use these results, together with lumpability properties and well-known variational principles, to obtain sharp capacity estimates that are key quantities in the proof. For a more detailed overview of the methods, we refer the reader to the monograph by Bovier and den Hollander~\cite{BdH15}. 

The remainder of the paper is organised as follows. Section~\ref{sec:prep} provides quantities and notations that are needed throughout the paper. Section~\ref{sec:profiles} identifies the metastable regime. Section~\ref{sec:Dir} provides a sharp approximation of the Dirichlet form associated with the Glauber dynamics in the presence of the disorder.  Section~\ref{sec:cap-val} provides estimates on capacity and on the metastable valley measure. Section~\ref{sec:proofs} proves Theorems~\ref{thm:metER}--\ref{thm:metER_lim}. Appendix~\ref{app:CW} contains a brief overview on known results for the standard CW model, which corresponds to the setting without disorder. Appendix~\ref{app:critical} gives numerical evidence for the presence of multiple metastable states for suitable choices of $\beta$, $h$ and of the law of the components of $J$. Appendix~\ref{app:hc_decr} contains an example in which $\beta \mapsto h_c(\beta)$ is not increasing, implying the possibility of a re-entrant metastable crossover. Appendix~\ref{app:prefactor} provides the limit as $n \to \infty$ of the prefactor in \eqref{CWERasymp}.


\section{Preparations}
\label{sec:prep}

Section~\ref{sec:Ham} introduces further notation and writes the Hamiltonian in terms of the level magnetisations. Section~\ref{sec:Diri} introduces the Dirichlet form associated with the Glauber dynamics and rewrites this in terms of the level magnetisations. Section~\ref{sec:GradHess} computes gradients and Hessians of the free energy as a function of the level magnetisations. Section~\ref{sec:Add} closes with an approximation of the free energy that will be needed later on.


\subsection{Hamiltonian}
\label{sec:Ham}

Recall \eqref{eq:Asets}. Abbreviate
\begin{equation}
\label{eq:def_omegaln}
\omega_{\ell,n} = \frac{\cAl}{n}.
\end{equation}
Since, by the law of large numbers, $(\omega_{\ell,n})_{\ell \in [k]} \to(\omega_{\ell})_{\ell \in [k]} \in (0,\infty)^k$ as $n\to\infty$ with $\cP$-probability tending to $1$, we may and will assume that $A_{\ell,n} \neq \emptyset$ for all $\ell \in [k]$ and all $n$ large enough. Recall \eqref{eq:levmagdef}--\eqref{eq:magvec}. Note that $m_{\ell,n}(\sigma)$ takes values in the set
\begin{equation}
	\Gamma_{\ell,n} =\left\{-1, -1 + \tfrac{2}{|A_{\ell,n}|}, \dots, 1-\tfrac{2}{|A_{\ell,n}|},1\right\}.
\end{equation} 
Hence $m_n(\sigma)$ takes values in the set
\begin{equation}
\label{eq:def_GP}
\Gamma_n = \bigtimes_{\ell \in [k]} \Gamma_{\ell,n}.
\end{equation}
The configurations corresponding to $M \subseteq \Gamma_n$ are denoted by
\begin{equation}
\label{eq:SnM}
\cS_n[M] = \{\si \in \cS_n \colon\, m_n(\si) \in M\}.
\end{equation}
For singletons $M=\{m\}$ we write $\cS_n[m]$ instead of $\cS_n[\{m\}]$.

Let
\begin{equation}
\label{eq:Hnalt}
H_n(\sigma) = - \frac{1}{2n} \sum_{i,j \in [n]} J(i) J(j)\, \sigma(i)\sigma(j) - h \sum_{i \in [n]} \sigma(i),
\qquad \sigma \in \cS_n,
\end{equation}
which is the Hamiltonian in \eqref{eq:Hn}, except for the diagonal term $-\frac{1}{2n} \sum_{i \in [n]} J^2(i)$, which is a constant shift. Using the notation above, we can write the Hamiltonian in \eqref{eq:Hnalt} as
\begin{equation}
\label{eq:H=nE}
H_n(\sigma) = -n \left[\frac12 \left(\sum_{\ell \in [k]} a_{\ell} \, \omega_{\ell, n}\,m_{\ell,n}  (\si) \right)^2 
+ h \sum_{\ell \in [k]} \omega_{\ell, n}\,m_{\ell,n} (\si) \right] = n E_n(m_n(\si)),
\end{equation}  
where we abbreviate
\begin{equation}
\label{eq:defE}
E_n(m) = -\frac12  \left(\sum_{\ell \in [k]}a_{\ell} \, \omega_{\ell, n}\, m_{\ell}  \right)^2 
-h \sum_{\ell \in [k]} \, \omega_{\ell, n} \, m_{\ell}, \qquad m=(m_{\ell})_{\ell \in [k]} \in \Gamma_n.
\end{equation}


\subsection{Dirichlet form and mesoscopic dynamics}
\label{sec:Diri}

By \eqref{eq:rate r}--\eqref{eq:mu}, the \emph{Dirichlet form} associated with the Glauber dynamics equals
\begin{equation}
\begin{split}
\cE_{\cS_n}(h,h) &= \frac{1}{2}\sum_{\sigma,\sigma' \in \cS_n} \mu_n(\sigma) r_n(\sigma,\sigma')\,[h(\sigma)-h(\sigma')]^2\\
&= \frac{1}{2 Z_n} \sum_{\sigma \in \cS_n} \sum_{\substack{\si' \in \cS_n, \\ \si' \sim \sigma}} \ee^{-\beta \, H_n(\sigma)}
\ee^{-\beta[H_n(\sigma') - H_n(\sigma)]_+} [h(\sigma)-h(\sigma')]^2,
\end{split}
\end{equation}  
where $h$ is a test function on $\cS_n$ taking values in $[0,1]$. Because of \eqref{eq:H=nE}, for any $h$ such that $h(\sigma) = \bar h(m_n(\sigma))$, with $\bar h$ a test function on $\Gamma_n$, we have 
\begin{equation}
\label{eq:Diri}
\begin{aligned}
\cE_{\cS_n}(h,h) 
&= \frac{1}{2Z_n} \sum_{m \in \Gamma_n} \sum_{m' \in \Gamma_n}  
\ee^{-\beta \, n E_n(m)} \ee^{-\beta \, n [E_n(m') - E_n(m)]_+} \\
&\qquad\qquad \times \big[\,\bar h(m)-\bar h(m')\big]^2 \sum_{\substack{\sigma \in \cS_n, \\ m_n(\si)=m}}\,\, 
\sum_{\substack{\si' \in \cS_n, \,\sigma' \sim \sigma, \\ m_n(\si')=m'}} 1,
\end{aligned}
\end{equation} 
where $m=(m_\ell)_{\ell \in [k]}$. If $\sigma' \sim \sigma$, then $\sigma' = \sigma^i$ for some $i \in [n]$, with $\sigma^i$ obtained from $\sigma$ by flipping the spin with label $i$. Let $\ell' \in [k]$ be such that $i \in A_{\ell',n}$. If $\si(i)=\pm1=-\si^{i}(i)$, then 
\begin{equation}
m_{\ell,n}(\si^{i}) =
\begin{cases}
m_{\ell',n}(\si) \mp \frac{2}{\abs{A_{\ell',n}}}, &\ell=\ell',\\
m_{\ell,n}(\si), &\ell \neq \ell'. 
\end{cases} 
\end{equation}

For $m, m' \in \Gamma_n$, we write $m \sim m'$ when there exists an $\ell' \in [k]$ such that $m'=m^{\ell',+}$ or $m'=m^{\ell',-}$, where
\begin{equation}
\label{eq:mpm}
m_{\ell}^{\ell',\pm} =
\begin{cases}
m_{\ell'} \pm \frac{2}{\abs{A_{\ell',n}}}, &\ell=\ell',\\
m_{\ell}, &\ell \neq \ell'. 
\end{cases}
\end{equation}
Moreover, for $\ell \in [k]$ and $\si \in \cS_n$ with $m_n(\si)=m$, the cardinality of the set $\{\si'\in\cS_n\colon\,\si' \sim \si,\,m_n(\si')=m^{\ell,\pm}\}$ equals $\frac{1\mp m_{\ell} }{2} |A_{\ell,n}|$, namely, the number of $(\mp 1)$-spins in $\si$ with index in $A_{\ell,n}$. Furthermore, 
\begin{equation}
\left \vert \{ \sigma \in \cS_n\colon\, m_n(\si) =m\}\right \vert
= \prod_{\ell \in [k]} \binom{|A_{\ell,n}|}{\frac{1+ m_{\ell}}{2} |A_{\ell,n}|}, 
\qquad m \in \Gamma_n,
\end{equation}
as is seen by counting the number of $(-1)$-spins with label in $A_{\ell,n}$ of a configuration with $\ell$-th level magnetisation $m_{\ell}$.  Using these observations, we can rewrite \eqref{eq:Diri} as
\begin{equation}
\label{eq:Diri2}
\begin{aligned}
&\cE_{\cS_n}(h,h) 
= \frac{1}{2 Z_n} \sum_{m \in \Gamma_n} \ee^{-\beta \, n E_n(m)} \sum_{m' \in \Gamma_n}  
\ee^{-\beta \, n [E_n(m') - E_n(m)]_+} \big[\,\bar h(m)-\bar h(m')\big]^2  \\
&\qquad \times \prod_{\ell \in [k]} \binom{|A_{\ell,n}|}{\frac{1+m_{\ell}}{2} |A_{\ell,n}|}  
\sum_{\ell \in [k]} |A_{\ell,n}|\left[\frac{1-m_{\ell}}{2}\,\mathds{1}(m'=m^{\ell,+}) 
+ \frac{1+m_{\ell}}{2}\, \mathds{1}(m'=m^{\ell,-})\right].
\end{aligned}
\end{equation}  

Next, abbreviate
\begin{equation}
\label{eq:defInEn}
I_n(m) = -\frac{1}{n} \log \left[\prod_{\ell \in [k]} \binom{|A_{\ell,n}|}{\frac{1+m_{\ell}}{2} |A_{\ell,n}|}\right],
\qquad m \in \Gamma_n,
\end{equation}
and put
\begin{equation}
\label{eq:defFn} 
F_n(m) = E_n(m) +\frac{1}{\beta} I_n(m)=-\frac12  \left(\sum_{\ell \in [k]}a_{\ell} \, \omega_{\ell, n}\, m_{\ell}  \right)^2 
-h \sum_{\ell \in [k]} \, \omega_{\ell, n} \, m_{\ell} +\frac{1}{\beta} I_n(m), \quad m \in \Gamma_n, 
\end{equation}
where $E_n(m)$ is defined in \eqref{eq:defE}. Moreover, define
\begin{equation}
\label{eq:brn}
\begin{aligned}
\bar r_n(m, m') &=  \ee^{-\beta n [E_n(m')-E_n(m)]_+} \\ 
&\qquad \times \sum_{\ell \in [k]} |A_{\ell,n}|
\left[\frac{1-m_{\ell}}{2}\, \mathds{1}(m'=m^{\ell,+}) 
+ \frac{1+m_{\ell}}{2}\, \mathds{1}(m'=m^{\ell,-})
\right].
\end{aligned}
\end{equation}
With this notation, we can write the \emph{mesoscopic measure} $\cQ_n(\cdot)= \mu_n \circ m_n^{-1}(\cdot)$ on $\Gamma_n$, with $\mu_n$ defined in \eqref{eq:mu}, as
\begin{equation}
\label{eq:Qm}
\mathcal{Q}_n(m)=\mu_n(\cS_n[m])
= \frac{1}{Z_n} \ee^{-\beta n F_n(m)},  \qquad m \in  \Gamma_n,
\end{equation}
and so the Dirichlet form in \eqref{eq:Diri2} becomes
\begin{equation}
\label{eq:Dirialt}
\cE_{\cS_n}(h,h) = \frac{1}{2}\sum_{m \in  \Gamma_n} \mathcal{Q}_n(m)\,
\sum_{m' \in \Gamma_n} \bar r_n(m, m')\,\big[\,\bar h(m)-\bar h(m')\big]^2.
\end{equation}  


\subsection{Gradients and Hessians}
\label{sec:GradHess}

Denote the Cram\'er entropy by 
\begin{equation}
I_{\mathbf{C}}(x) = \frac{1-x}{2} \log \left(\frac{1-x}{2}\right)  
+ \frac{1+x}{2} \log \left(\frac{1+x}{2}\right).
\end{equation}
Define
\begin{equation}
\label{eq:barIn}
\bar{I}_n(m) = \sum_{\ell \in [k]}\omega_{\ell,n} I_{\mathbf{C}}(m_{\ell}).
\end{equation}
Since $\cAl = [1+o_n(1)]\,\omega_\ell n$, we can use Stirling's formula $N! = [1+o_N(1)]\,N^N\ee^{-N} \sqrt{2\pi N}$ to obtain 
\begin{equation}
\label{eq:In}
I_n(m) = \bar{I}_n(m) + \sum_{\ell \in [k]} \frac{1}{2n}\log \left(\frac{\pi (1-m_{\ell}^2) |A_{\ell,n}|}{2}\right)  
+ o(n^{-1}) = \bar{I}_n(m) + O(n^{-1}\log n),
\end{equation}
where the error term is \emph{uniform} in $m \in  \Gamma_n$. For $\ell, \bar{\ell} \in [k]$, we compute
\begin{equation}
\frac{\partial \bar{I}_n(m)}{\partial m_{\ell}} =  \frac{\omega_{\ell,n}}{2}  \log \left(\frac{1+m_{\ell}}{1-m_{\ell}}\right)
\end{equation}
and
\begin{equation}
\begin{aligned}
\frac{\partial^2 \bar{I}_n(m)}{\partial {m_{\ell}}\partial {m_{\bar{\ell}}}}
&= 0, \qquad \ell \neq \bar{\ell},\\
\frac{\partial^2 \bar{I}_n(m)}{\partial {m_{\ell}}^2}
&= \frac{\omega_{\ell,n}}{1-m_{\ell}^2}.
\end{aligned}
\end{equation}
Recalling \eqref{eq:defE}, we compute 
\begin{equation}
\begin{split}
\frac{\partial E_n(m) }{\partial m_{\ell}}
&= -a_{\ell} \, \omega_{\ell,n} \left(\sum_{\ell' \in [k]} a_{\ell'} \omega_{\ell',n}\, m_{\ell'} \right) -  \omega_{\ell,n}h.
\end{split}
\end{equation}
Define
\begin{equation}
\label{eq:barFn}
\bar{F}_n(m) = E_n(m) +\frac{1}{\beta} \bar{I}_n(m) = -\frac{1}{2} \left(\sum_{\ell \in [k]} a_{\ell} \, 
\omega_{\ell,n} \, m_{\ell}\right)^2  - h \sum_{\ell \in [k]}  \omega_{\ell,n} m_{\ell} +\frac{1}{\beta} \bar{I}_n(m).
\end{equation}

\begin{remark}
\label{rem:FnbarFn}
{\rm By \eqref{eq:In}, $F_n(m)= \bar{F}_n(m) + O(n^{-1}\log n)$, where $F_n$ is defined in \eqref{eq:defFn}.}\hfill$\spadesuit$
\end{remark}

For $m \in [-1,1]^k$, define
\begin{equation}
\begin{split}
\label{eq:F}
F_{\beta,h}(m)
&= -\frac12  \left(\sum_{\ell \in [k]}a_{\ell} \, \omega_{\ell}\, m_{\ell}  \right)^2 
-h \sum_{\ell \in [k]} \, \omega_{\ell} \, m_{\ell} 
+\frac{1}{\beta}  \sum_{\ell \in [k]}\omega_{\ell} 	I_{\mathbf{C}}(m_{\ell}),
\end{split}
\end{equation}
which corresponds to the uniform limit in probability of $F_n$ as $n \to \infty$. Compute
\begin{equation}
\label{eq:GradFn}
\frac{\partial\bar{F}_n(m)}{\partial m_{\ell}} 
= \omega_{\ell,n}\left[\frac{1}{2\beta}  \log \left(\frac{1+m_{\ell}}{1-m_{\ell}}\right) 
- a_{\ell}\left(\sum_{\ell' \in [k]} a_{\ell'} \, \omega_{\ell',n} \,m_{\ell'} \right) -h\right]
\end{equation}
and
\begin{equation}
\label{eq:Hess_barF}
\begin{aligned}
\frac{\partial^2 \bar{F}_n(m)}{\partial {m_{\ell}}\,\partial {m_{\ell'}}}
&= -a_{\ell} \, \omega_{\ell,n} \,a_{\ell'} \,\omega_{\ell',n}, \qquad \ell \neq \ell',\\
\frac{\partial^2 \bar{F}_n(m)}{\partial {m_{\ell}}^2}
&= \frac{\omega_{\ell,n}}{\beta}\frac{1}{1-m_{\ell}^2} -a_{\ell}^2\, \omega_{\ell,n}^2.
\end{aligned}
\end{equation}
The same formulas apply for $I_n,F_n$, with an error term $O(n^{-1})$. 


\subsection{Additional computation}
\label{sec:Add}

We conclude with a computation that will be useful later on. Recalling \eqref{eq:mpm}, we write
\begin{equation}
\begin{split}%
&n \big[\bar{I}_n(m^{\ell,\pm})-\bar{I}_n(m)\big]\\
&= n \, \omega_{\ell,n} \Bigg[\frac{1+m_{\ell}}{2} \log \left(1\pm\frac{2}{\cAl (1+m_{\ell})}\right) 
+ \frac{1-m_{\ell}}{2} \log \left(1\mp\frac{2}{\cAl (1-m_{\ell})}\right) \pm \frac{1}{\cAl }A_{\ell,n}^\pm \Bigg]\\
&= n \, \omega_{\ell,n} \left[\pm\frac{1}{\cAl } \mp\frac{1}{\cAl} + O(n^{-2}) 
\pm \frac{1}{\cAl} \Delta_{\ell,n}^\pm \right] =  \Delta_{\ell,n}^\pm + O(n^{-1}),
\end{split}
\end{equation}
where
\begin{equation}\label{eq:defDelta}
\Delta_{\ell,n}^\pm = \log \left(1+ \frac{2 m_{\ell} \pm \tfrac{4}{\cAl}}{1-m_{\ell}\mp \tfrac{2}{\cAl}}\right).
\end{equation}
The same formula applies for $I_n$ with an error term of order $O(n^{-1})$, and hence 
\begin{equation}
\label{eq:In-In}
n \big[{I}_n(m^{\ell,\pm})-{I}_n(m)\big] = \Delta_{\ell,n}^\pm + O(n^{-1}).
\end{equation}
Note that $\Delta_{\ell,n}^\pm=O(1)$. Therefore, using \eqref{eq:defFn}, we get
\begin{equation}
\label{eq:nE-nE}
\begin{split}
n\left[E_n(m^{\ell,\pm})-E_n(m)\right]
&= n\left[ F_n(m^{\ell,\pm})-F_n(m) \right] -\frac{1}{\beta}n \left[ I_n(m^{\ell,\pm})-I_n(m) \right]\\
&= n\left[ F_n(m^{\ell,\pm})-F_n(m) \right] -\frac{1}{\beta} \Delta_{\ell,n}^\pm  + O(n^{-1}).
\end{split}
\end{equation}


\section{Metastable regime}
\label{sec:profiles}

Section~\ref{sec:critF} identifies the stationary points of $\bar{F}_n$. Section~\ref{sec:metreg} identifies the metastable regime.
Section~\ref{sec:1dim} provides details on the $1$-dimensional metastable landscape. 


\subsection{Stationary points of $\bar{F}_n$ and $F_{\beta,h}$}
\label{sec:critF}

By \eqref{eq:GradFn}, the critical points $m=(m_{\ell})_{\ell \in [k]}$ of $\bar{F}_n$ solve the system of equations (with $\omega_{\ell,n}\neq 0$)
\begin{equation}
0 = \frac{\partial  \bar{F}_n(m)}{\partial m_{\ell}} 
= \omega_{\ell,n}\left[\frac{1}{2 \beta }  \log \left(\frac{1+m_{\ell}}{1-m_{\ell}}\right) 
-a_{\ell}\left(\sum_{\ell' \in [k]}a_{\ell'} \, \omega_{\ell',n} \,m_{\ell'} \right) -h\right], \quad \ell \in [k].
\end{equation}
Hence
\begin{equation}
\label{eq:crit11}
\frac{1}{2} \log \left(\frac{1+m_{\ell}}{1-m_{\ell}}\right) 
= \beta  \left[a_{\ell}\left(\sum_{\ell' \in [k]} a_{\ell'}\,\omega_{\ell',n}\,m_{\ell'}\right)+h\right].
\end{equation}
Since $\arctanh x = \tfrac{1}{2} \log \frac{1+x}{1-x}$, $x \in (-1,+1)$, \eqref{eq:crit11} can be rewritten as
\begin{equation}
\label{eq:crit}
m_{\ell}=\tanh\left(\beta \left[a_{\ell}\left(\sum_{\ell' \in [k]} a_{\ell'}\,\omega_{\ell',n}\,m_{\ell'}\right)+h\right]\right),
\qquad \ell \in [k].
\end{equation}
Similarly, the critical points $m=(m_{\ell})_{\ell \in [k]}$ of $F_{\beta,h}$ solve the deterministic equation
\begin{equation}
\label{eq:m*}
m_{\ell} = \tanh \left(\beta\left[a_{\ell}\left(\sum_{\ell' \in [k]} a_{\ell'} \, \omega_{\ell'} \,m_{\ell'} \right) + h \right]\right),
\qquad \ell \in [k].
\end{equation}
Note that this can also be obtained directly from \eqref{eq:crit} after replacing $\omega_{\ell,n}$ by its mean value $\omega_{\ell}$.


\subsection{Metastable regime} 
\label{sec:metreg} 

We are interested in identifying the metastable regime, i.e., the set of pairs $(\beta,h)$ for which $F_{\beta,h}$ has more than one minimum. Put 
\begin{equation}
\label{eq:defKm}
K = K(m)=\sum_{\ell \in [k]} a_{\ell}\,\omega_{\ell}\,m_{\ell}.
\end{equation} 
From the characterisation of the critical points of $F_{\beta,h}$ in \eqref{eq:m*} it follows that
\begin{equation}
\label{eq:Keq}
K = T_{\beta,h}(K), \qquad T_{\beta,h}(K)
= \sum_{\ell \in [k]} a_{\ell}\,\omega_{\ell}\,\tanh(\beta[a_{\ell}K+h]).
\end{equation}
Note that any critical point $m=(m_{\ell})_{\ell \in [k]} \in [-1,1]^k$ of $F_{\beta,h}$ is uniquely determined by $K(m )\in \R$. Consequently, the problem of solving the $k$-dimensional system in \eqref{eq:m*} can be reduced to solving the $1$-dimensional equation \eqref{eq:Keq}. Recalling Hypothesis~\ref{hyp}(\ref{hyp:non-deg}), the system is in the metastable regime if and only if \eqref{eq:Keq} has more than one solution that is not tangent to the diagonal.

Compute
\begin{equation}
\label{eq_T'}
\begin{aligned}
T'_{\beta,h}(K) &= \beta \sum_{\ell \in [k]} a^2_{\ell}\,\omega_{\ell}\,\big(1-\tanh^2(\beta[a_{\ell}K+h])\big),\\
T''_{\beta,h}(K) &= - 2\beta^2 \sum_{\ell \in [k]} a^3_{\ell}\,\omega_{\ell}\,\tanh(\beta[a_{\ell}K+h])\,
\big(1-\tanh^2(\beta[a_{\ell}K+h])\big).
\end{aligned}
\end{equation}
For $h=0$, the system is metastable when
\begin{equation}
\label{eq:beta_regime}
\beta > \frac{1}{\sum_{\ell \in [k]} a_{\ell}^2  \, \omega_{\ell}},
\end{equation}
in which case $T_{\beta,h}$ has a unique inflection point at $K=0$, implying that \eqref{eq:Keq} has three solutions $K \in \{-K^*, 0,+K^*\}$ with $K^*>0$. Otherwise \eqref{eq:Keq} has only one solution $K = 0$.

We proceed with the more interesting case $h>0$.

\subsubsection{Number of solutions}

\begin{lemma}[{\bf Number of solutions}]
For $h>0$, the number of critical points of $F_{\beta,h}$, i.e., solutions of \eqref{eq:Keq}, varies in $\{1,3, \dots, 2\ell+1\}$, where $\ell \in [k]$ and $2\ell-1$ is the number of inflection points of $T_{\beta,h}$.
\end{lemma}

\begin{proof}
For $h>0$ and $K$ positive and large enough, $T_{\beta,h}''(K)<0$. Moreover, for $h>0$ and $K$ negative with $\abs{K}$ large enough, $T_{\beta,h}''(K)>0$. Therefore, $T_{\beta,h}$ has at least one inflection point and that the number of inflection points of $T_{\beta,h}$ cannot be even: it takes values in $\{1,3, \dots, 2k-1\}$ depending on $\beta,h$ and the law of the components of $J$. Consequently, if $2\ell-1$ ($\ell \in [k]$) is the number of inflection points, then the cardinality of the solutions of \eqref{eq:Keq} takes values in $\{1,3, \dots, 2\ell+1\}$ depending on $\beta,h$ and on the law of the components of $J$.
\end{proof}

We conjecture that for any finite $k$ there exist $\beta,h$ and a law of the components of $J$ such that \eqref{eq:Keq} has any number of solutions in the set $\{1, 3, \dots, 2k+1\}$. We found numerical evidence for this fact for $k \in \{2,3,4\}$. See Appendix~\ref{app:critical}.

\begin{lemma}[{\bf Unique strictly positive solution}]
\label{lem:1pos}
$\mbox{}$
For every $\beta>0$ and $h>0$, \eqref{eq:Keq} has exactly one strictly positive solution.
\end{lemma} 

\begin{proof}
Put $W(K)=T_{\beta,h}(K)-K$. The solutions of $\eqref{eq:Keq}$ are the roots of $W$. Clearly, $W(0)>0$. Moreover, $\lim_{K\to\infty} W(K) = -\infty$ because $\lim_{K\to \infty} T_{\beta,h}(K) = \sum_{\ell \in [k]} a_{\ell}\,\omega_{\ell}>0$ is finite. Therefore, by continuity, a root of $W(K)$ exists in $(0,\infty)$. 
	
Let $\tilde{K}$ be the smallest positive root of $W$. Next we will prove that this root is unique. Indeed, $W(K)''<0$ when $K \in [0, \infty)$, meaning that $K \mapsto W(K)'$ is strictly decreasing. By continuity, since $W(K)>0$ for all $K \in [0, \tilde{K})$, we have $W(\tilde{K})'\leq 0$ and $\lim_{K\to\infty}W(K)'=-1$. Therefore, $W(K)'< 0$ for all $K \in (\tilde{K}, \infty)$, and so $W$ is strictly decreasing. Moreover,  $W(K)<W(\tilde{K})=0$ for all $K \in (\tilde{K},\infty)$. Thus, $\tilde{K}$ is the only positive root of $W$.
\end{proof}


\subsubsection{Metastable regime}

\begin{lemma}[{\bf Characterisation of the metastable regime}]
\label{lem:3sol}
$\mbox{}$\\
\eqref{eq:Keq} has at least three solutions not tangent to the diagonal if and only if there exists $\bar{K}<0$ such that $\bar{K}>T_{\beta,h}(\bar{K})$, i.e.,
\begin{equation}
\label{eq:critical}
\begin{split}
&\bar{K} > \sum_{\ell \in [k]} a_{\ell}\,\omega_{\ell}\,\tanh(\beta[a_{\ell}\bar{K}+h]).
\end{split}
\end{equation}
\end{lemma}

\begin{proof}
Using Lemma~\ref{lem:1pos}, we see that \eqref{eq:Keq} has at least three solutions if and only if it has at least two strictly negative solutions. As above, we define $W(K)=T_{\beta,h}(K)-K$. The solutions of \eqref{eq:Keq} are the roots of $W$.	 Now, \emph{assume} that there exists a $\bar{K}<0$ such that $\bar{K}>T_{\beta,h}(\bar{K})$. Since $W(\bar{K})<0$ and $W(0)>0$, $W(K)$ has a root in $(\bar{K},0)$, implying that \eqref{eq:Keq} has at least one solution in $(\bar{K},0)$. Moreover, since $\lim_{K\to - \infty} T_{\beta,h}(K)= -\sum_{\ell \in [k]} a_{\ell}\,\omega_{\ell}$ is finite, we have $\lim_{K\to - \infty} W(K)= \infty$. Because $W(\bar{K})<0$, it follows that $W$ has at least one root in $(-\infty,\bar{K})$. With the same argument it can be shown that the negative roots of $W$ are always even. The opposite implication is trivial.
\end{proof}

\begin{remark}
\label{rem:T'1}
{\rm Applying the intermediate value theorem to the derivative of $W(K)=T_{\beta,h}(K)-K$, we get that if the condition in Lemma~\ref{lem:3sol} is satisfied, then there exists a $\bar{K}<0$ such that $T_{\beta,h}'(\bar{K})= 1$ and $\bar{K}>T_{\beta,h}(\bar{K})$.} \hfill$\spadesuit$
\end{remark}

\begin{theorem}[{\bf Metastable regime}]
\label{thm:meta_reg}
Define, as in \eqref{eq:crt},
\begin{equation}
\beta_c=\frac{1}{\sum_{\ell \in [k]} a^2_{\ell}\,\omega_{\ell}}.
\end{equation}
The metastable regime is 
\begin{equation}
\label{eq:meta_reg}
\beta \in (\beta_c,\infty), \qquad h \in \big[0,h_c(\beta)\big),
\end{equation}
with $\beta \mapsto \beta h_c(\beta)$ non-decreasing on $[\beta_c,\infty)$. Furthermore, if the support of the law of the components of $J$ is put into increasing order, i.e., $a_1<a_2< \dots <a_k$, then
\begin{equation} 
\label{eq:hcLim}
\lim_{\beta \to \infty} h_c(\beta) = \min_{\ell \in [k]^*} \left(\sum_{\ell' =\ell}^k a_{\ell}\, a_{\ell'} \, 
\omega_{\ell'} - \sum_{\ell' =1}^{\ell -1} a_{\ell} \,a_{\ell'} \,  \omega_{\ell'}\right),
\end{equation}
where the minimum is over all $\ell \in [k]$ such that the quantity between brackets is positive.
\end{theorem}

\begin{proof}
Recalling Lemma~\ref{lem:3sol}, we look for conditions for the existence of a $K<0$ satisfying \eqref{eq:critical}. If such a $K$ exists, then by Remark~\ref{rem:T'1} there exists a $\bar{K}<0$ satisfying \eqref{eq:critical} such that $T_{\beta,h}'(\bar{K})= 1$, which reads
\begin{equation}
\label{eq:T'1}
\sum_{\ell \in [k]} a^2_{\ell}\,\omega_{\ell}\,\tanh^2(\beta[a_{\ell}\bar{K}+h]) 
= \sum_{\ell \in [k]} a^2_{\ell}\,\omega_{\ell} -\frac{1}{\beta}.
\end{equation}
Since the left-hand side of \eqref{eq:T'1} is positive, it admits solutions only if
\begin{equation}
\label{eq:1/b}
\frac{1}{\beta}<\sum_{\ell \in [k]} a^2_{\ell}\,\omega_{\ell}= \frac{1}{\beta_c}.
\end{equation}
Therefore, \eqref{eq:1/b} is a necessary condition for the metastable regime.
	
Now assume \eqref{eq:1/b}. Since $\tanh x \sim x$, $x \to 0$, for $\abs{K}\ll \beta (\max_{\ell \in [k]}a_{\ell})^{-1}$ and $h\downarrow 0$, we have
\begin{equation}
\begin{split}
K&=T_{\beta,h}(K)=\sum_{\ell \in [k]} a_{\ell}\,\omega_{\ell}\,\tanh(\beta[a_{\ell}K+h]) 
\sim \sum_{\ell \in [k]} a_{\ell}\,\omega_{\ell}\,\beta[a_{\ell}K+h],
\end{split}
\end{equation}
which reads
\begin{equation}
K \sim -\left(\sum_{\ell \in [k]} a_{\ell}\,\omega_{\ell}\right) \left( \frac{1}{\beta_c} - \frac{1}{\beta}\right)^{-1} h
\end{equation}
and proves the existence of a negative solution. A positive solution is guaranteed by Lemma \ref{lem:1pos}. The existence of a third (strictly negative) solution of \eqref{eq:m*}, for every $\beta>\beta_c$ and for $h\downarrow 0$, follows as in the proof of Lemma~\ref{lem:3sol}. Therefore, the lower bound on $\beta_c$ is sharp.
	
Since $h \mapsto T_{\beta,h}(K)$ is strictly increasing for every fixed $\beta>0$ and $K\in \R$, there exists a unique critical curve $\beta \mapsto h_c(\beta)$ such that the system is metastable for $0 \leq h < h_c(\beta)$ and not metastable for $h \geq h_c(\beta)$. We know that $h_c(\beta) > 0$ for $\beta > \beta_c$. By passing to the parametrisation $g = h\beta$, we get that $\beta \mapsto T_{\beta,g}(K)$ is strictly decreasing for every $g$ and for every $K<0$, from which it follows that $\beta \mapsto g_c(\beta)=\beta h_c(\beta)$ is non-decreasing.
	
We next focus on the limit of $h_c(\beta)$ as $\beta\to\infty$. By Lemma~\ref{lem:3sol}, we may focus on the existence of $\bar{K}$ satisfying \eqref{eq:critical}. In the limit as $\beta\to\infty$, $\tanh(\beta[a_{\ell}\bar{K}+h])\to 2\Theta_{-h/a_{\ell}}(\bar{K})-1$, where $\Theta_{x}(\cdot)$ is the Heaviside function centred in $x$. Thus, for all $\ell \in [k+1]$,
\begin{equation}
\lim_{\beta \to \infty}\sum_{\ell' \in [k]} a_{\ell'}\,\omega_{\ell'}\,\tanh(\beta[a_{\ell'}K+h]) 
=  -\sum_{\ell' =\ell}^k a_{\ell'} \, \omega_{\ell'} + \sum_{\ell' =1}^{\ell -1} a_{\ell'} \, \omega_{\ell'}, 
\qquad K \in \left(-\frac{h}{a_{\ell-1}}, -\frac{h}{a_{\ell}}\right),
\end{equation}
and, for all $\ell \in [k]$,
\begin{equation}
\lim_{\beta \to \infty}\sum_{\ell' \in [k]} a_{\ell'}\,\omega_{\ell'}\,\tanh(\beta[a_{\ell'}K+h]) 
=  -\sum_{\ell' =\ell+1}^k a_{\ell'} \, \omega_{\ell'} + \sum_{\ell' =1}^{\ell -1} a_{\ell'} \, \omega_{\ell'}, 
\qquad K = -\frac{h}{a_{\ell}},
\end{equation}
where we set $-\frac{h}{a_{0}}=-\infty$ and $-\frac{h}{a_{k+1}}=\infty$. Thus, for $\bar{K} \in \left(-\frac{h}{a_{\ell-1}}, -\frac{h}{a_{\ell}}\right)$, \eqref{eq:critical} can be written as
\begin{equation}
\bar{K} > -\sum_{\ell' =\ell}^k a_{\ell'} \, \omega_{\ell'} + \sum_{\ell' =1}^{\ell -1} a_{\ell'} \,  \omega_{\ell'}.
\end{equation}
Therefore, \eqref{eq:critical} has a solution if and only if there exists an $\ell \in [k]$ such that
\begin{equation}
\label{eq:iff_h}
-\sum_{\ell' =\ell}^k a_{\ell'} \, \omega_{\ell'} + \sum_{\ell' =1}^{\ell -1} a_{\ell'} \,  \omega_{\ell'} < -\frac{h}{a_{\ell}},
\end{equation}
in which case a solution $\bar{K}$ of \eqref{eq:critical} exists in
\begin{equation}
\label{eq:interv}
\left(-\sum_{\ell' =\ell}^k a_{\ell'} \, \omega_{\ell'} + \sum_{\ell' =1}^{\ell -1} a_{\ell'} \, 
\omega_{\ell'}, - \frac{h}{a_{\ell}}\right).
\end{equation}
Note that the quantity between brackets in \eqref{eq:hcLim} is always positive for $\ell=1$. Thus, the minimum is always finite. 
	
The proof is complete after we show why we may drop the case where $\bar{K}  = -\frac{h}{a_{\ell}}$ for some $\ell \in [k]$. In this case the condition for $\bar{K}$ to satisfy \eqref{eq:critical} is 
\begin{equation}
\label{eq:iff_h2}
-\sum_{\ell' =\ell+1}^k a_{\ell'} \, \omega_{\ell'} + \sum_{\ell' =1}^{\ell -1} a_{\ell'} \,  \omega_{\ell'} < -\frac{h}{a_{\ell}},
\end{equation}
which implies \eqref{eq:iff_h}. Thus, if $\bar{K}  = \frac{-h}{a_{\ell}}$ satisfies \eqref{eq:critical}, then also some other $K$ in \eqref{eq:interv} satisfies \eqref{eq:critical}. Therefore, the condition in \eqref{eq:iff_h} is equivalent to having metastability.
\end{proof}

\begin{lemma}[{\bf Re-entrant crossover}]
The function $\beta \mapsto h_c(\beta)$ is not necessarily non-decreasing. 
\end{lemma}

\begin{proof}
In Appendix~\ref{app:hc_decr} we provide an example of $\beta \mapsto h_c(\beta)$ that is not increasing.
\end{proof}


\subsubsection{Bounds on the inflection points and on the critical curve}

\begin{lemma}[{\bf Bounds on inflection points}] 
\label{lem:inflection}
All solutions of $T_{\beta,h}''(K)= 0$ are contained in the interval  
\begin{equation}
\left[-\frac{h}{\min_{\ell \in [k]}a_{\ell}}, - \frac{h}{\max_{\ell \in [k]}a_{\ell}}\right].
\end{equation}
In particular, they are all strictly negative.
\end{lemma}

\begin{proof}
If $K>-\frac{h}{\max_{\ell \in [k]}a_{\ell}}$, then $\tanh(\beta[a_{\ell}K+h])>0$ for all $\ell \in [k]$, which implies $T_{\beta,h}''(K)<0$. If $K<-\frac{h}{\min_{\ell \in [k]}a_{\ell}}$, then $\tanh(\beta[a_{\ell}K+h])<0$ for all $\ell \in [k]$, which implies $T_{\beta,h}''(K)>0$. 
\end{proof}

\begin{lemma}[{\bf Upper bound on $h_c$}]
\label{lem:upper_h}
$\sup_{\beta \in (\beta_c,\infty)} h_c(\beta) < \left(\max_{\ell \in [k]} a_{\ell}\right) \sum_{\ell \in [k]} a_{\ell}\,\omega_{\ell}$.
\end{lemma}

\begin{proof}
Use Lemma~\ref{lem:3sol} to characterise the metastable regime and  Remark~\ref{rem:T'1}. We \emph{claim} that if a solution $\bar{K}$ of \eqref{eq:critical} with $T_{\beta,h}'(\bar{K})=1$ exists, then it must be negative and strictly less than an inflection point. Using this fact, together with Lemma~\ref{lem:inflection} and the inequality in \eqref{eq:critical}, we obtain a necessary upper bound on $h$:
\begin{equation}
\begin{split}
\sum_{\ell \in [k]} a_{\ell}\,\omega_{\ell}\,\tanh(\beta[a_{\ell}\bar{K}+h])
< -\frac{h}{\max_{\ell \in [k]}a_{\ell}}.
\end{split}
\end{equation}
Using that $\tanh(\beta[a_{\ell}\bar{K}+h])>-1$, we conclude the proof. 
	
We are left to prove the claim. By Lemma~\ref{lem:inflection}, all inflection points are negative, and $T_{\beta,h}''(K)<0$ for $K\geq 0$. Assume, by contradiction, that $T_{\beta,h}''(K)<0$ for all $K \in (\bar{K}, \infty)$. Then $T_{\beta,h}'$ is strictly decreasing. Therefore, $T_{\beta,h}'(K)<1$ for all $K \in (\bar{K}, \infty)$, which implies $T_{\beta,h}(K)-T_{\beta,h}(0)<K$. Since $T_{\beta,h}(0)>0 $, there exists a $\tilde{K} \in (\bar{K},0)$ such that $T_{\beta,h}(\tilde{K} )>0>\tilde{K}$. Thus, $T_{\beta,h}(\tilde{K})-T_{\beta,h}(0)>\tilde{K}$, which contradicts what we have proved for all $K \in (\bar{K}, \infty)$.
\end{proof}


\subsection{Quasi 1-dimensional landscape} 
\label{sec:1dim}

Given $K \in \R$, by standard saddle point approximation, the leading order of
\begin{equation}
-\frac{1}{\beta n}\log 	\mu_n\big(\{\si\colon\, K_n(m_{n}(\si))=K\}\big)
\end{equation}
turns out to be the function $G_n \colon \R \to \R$ defined by
\begin{equation}
G_n(K) = \inf_{m\colon\, K_n(m)=K} \bar{F}_n(m).
\end{equation}
Recalling definitions \eqref{eq:barFn} and \eqref{eq:defKm}, using Lagrange multipliers and integrating the condition $K_n(m)=K$, we obtain 
\begin{equation}
\label{eq:Gn}
G_n(K) =-\frac{1}{2}K^2 -\frac{\log 2}{\beta} - \inf_{t \in \R} \left( Kt 
+ \sum_{\ell \in [k]} \frac{\oln}{\beta}\log\cosh\left[\beta (h-t a_{\ell})\right]\right).
\end{equation}

\begin{lemma}[{\bf Alternative characterisation for the critical points}]
\label{lem:G}
$\mbox{}$
\begin{enumerate}
\item 	
If $m^*$ is a (not maximal) critical point for $F_{n}$, then $K_n(m^*)$ is a critical point for $G_n$.
\item 
If $K$ is a critical point for $G_{n}$, then $m^*=(m^*_{\ell})_{\ell \in [k]}$ with $m^*_{\ell} = \tanh\left(\beta \left[a_{\ell}K+h\right]\right)$ (recall \eqref{eq:crit}) is a critical point for $F_n$.
\item 
$F_n(m^*)=G_n(K_n(m^*))$ for any (not maximal) critical point $m^*$. 
\label{item:3}
\end{enumerate}
\end{lemma}

\begin{proof}
Similar to \cite[Lemma 7.4]{BEGK1}. 
\end{proof}

We have already seen that $K_n(m)$ fully determines any critical value $m$ of $F_n$, and is useful to order them. Lemma~\ref{lem:G} exhibits the one-dimensional structure underlying the metastable landscape and provides a tool to describe the nature of the critical points of $F_n$.

\begin{remark}
{\rm The above results extend to the limit $n \to \infty$: replace $F_n$ by $F_{\beta,h}$ and $G_n$ by $G_{\beta,\ell}$, obtained after replacing $\oln$ by $\omega_{\ell}$ in \eqref{eq:Gn}, and $K_n(\cdot)$ by $K(\cdot)$.}\hfill$\spadesuit$
\end{remark}


\section{Approximation of the Dirichlet form near the saddle point}
\label{sec:Dir}

In this section we approximate the Dirichlet form associated with the coarse-grained dynamics near the saddle point. This is a key step to obtain capacity estimates in the following section. Further details and examples on the techniques we use here can be found in \cite[Chapters 9, 10 and 14]{BdH15}.

Section~\ref{sec:keyobj} introduces some key quantities that are needed to express the mesoscopic measure. Section~\ref{sec:appdyn} introduces an approximate mesoscopic measure that leads to an approximate dynamics. Section~\ref{sec:appharm} approximates the harmonic functions associated with this dynamics. Section~\ref{sec:appDiri} computes an approximate Dirichlet form. Section~\ref{sec:finalDiri} uses the latter to approximate the full Dirichlet form.


\subsection{Key quantities}
\label{sec:keyobj}

Let $\mathbf{m}_n = (\mathbf{m}_{\ell,n})_{\ell \in [k]}$ and $\mathbf{t}_n = (\mathbf{t}_{\ell,n})_{\ell \in [k]}$  in $\Gamma_n$ be a local minimum of $F_n$ and the correspondent saddle point, respectively, as defined in Section~\ref{sec:crossover}. Note that both  $\mathbf{m}_n$ and $\mathbf{t}_n$ satisfy \eqref{eq:crit}. Consider the neighbourhood of $\mathbf{t}_n$ defined by  
\begin{equation}
\label{eq:defDn}
\cD_n=\left\{ m \in \Gamma_n\colon\, d(m, \mathbf{t}_n) \leq C'n^{-1/2} \log^{1/2} n \right\},
\end{equation}
where $d$ is the Euclidean distance and $C' \in (0,\infty)$ is a constant. Abbreviate the Hessian of $F_n$
\begin{equation}
\label{eq:defAnm}
\mathbb{A}_n (m)= (\nabla^2 F_n)(m), \qquad m \in \Gamma_n,
\end{equation} 
and put
\begin{equation}
\label{eq:An}
\mathbb{A}_n = \mathbb{A}_n(\mathbf{t}_n).
\end{equation} 
By \eqref{eq:Hess_barF}, 
\begin{equation}
\label{eq:Anm}
\begin{aligned}
(\mathbb{A}_n(m))_{\ell, \ell'}
&= -a_{\ell} \, \omega_{\ell,n} \,a_{\ell'} \,\omega_{\ell',n} + O(n^{-1}),
\qquad \ell \neq \ell',\\
(\mathbb{A}_n(m))_{\ell, \ell}
&= \frac{\omega_{\ell,n}}{\beta}\frac{1}{1-m_{\ell}^2} -a_{\ell}^2\, \omega_{\ell,n}^2 + O(n^{-1}) 
=  \frac{1}{\beta}\frac{\partial^2 \bar{I}_n(m)}{\partial {m_{\ell}}^2} -a_{\ell}^2\, \omega_{\ell,n}^2 + O(n^{-1}).
\end{aligned}
\end{equation}
Note that $\mathbb{A}_n(m)$ is a diagonal matrix minus a rank one matrix. Compute
\begin{equation}
\label{eq:detAm}
\begin{split}
\det \mathbb{A}_n(m)
&= \left(1-\sum_{\ell \in [k]} \beta \, a_{\ell}^2\, \omega_{\ell,n} [1-m_{\ell}^2]  \right) 
\prod_{\substack{\ell' \in [k]}}  \frac{1}{\beta} \frac{\omega_{\ell',n}}{1-m_{\ell'}^2} [1+O(n^{-1})]. 
\end{split}
\end{equation}


\subsection{Approximate dynamics and Dirichlet form}
\label{sec:appdyn}
For any two vectors $\mathbf{v},\mathbf{w} \in \R^k$, let $\vprod{\mathbf{v}}{\mathbf{w}}$ denote their scalar product. For any $k\times k$ matrix $\mathbf{M}$ and any $\mathbf{v} \in \R^k$, let $\mathbf{M} \cdot \mathbf{v}$ denote their matrix product, as $\mathbf{v}$ was in $\R^k\times 1$.

For $m \in \cD_n$, define 
\begin{equation}
\label{eq:tiQ}
\tilde{\mathcal{Q}}_n(m)=\frac{1}{Z_n}\exp \left[-\tfrac{\beta n}{2}\big\langle [m-\mathbf{t}_n],\mathbb{A}_n\cdot[m-\mathbf{t}_n]\big\rangle\right] \exp\left[-\beta n F_n(\mathbf{t}_n)\right],
\end{equation}
and 
\begin{equation}
\label{eq:trn}
\begin{aligned}
&\tilde{r}_n\big(m,m'\big) = 
&\begin{cases}
\bar r_n\big(\mathbf{t}_n,\mathbf{t}_n^{\ell,+}\big), 
&m'=m^{\ell,+},\\[0.2cm]
\bar r_n\big(\mathbf{t}_n^{\ell,-},\mathbf{t}_n\big)
\frac{\tilde{\mathcal{Q}}_n(m^{\ell,-})}{\tilde{\mathcal{Q}}_n(m)},
&m'=m^{\ell,-},\\[0.2cm]
0, 
&\text{else,}
\end{cases}
\end{aligned}
\end{equation}
where $\bar{r}_n$ is defined in \eqref{eq:brn}. The transition rates $\tilde{r}_n$ define a random dynamics on $\cD_n$ that is reversible with respect to the mesoscopic measure $\tilde{\mathcal{Q}}_n$. The corresponding \emph{Dirichlet form} is
\begin{equation}
\label{eq:ti_cE}
\tilde{\cE}_{\cD_n}(u,u)
=\sum_{m \in \cD_n} \tilde{\mathcal{Q}}_n(m) 
\sum_{\ell \in [k]} \tilde{r}_n\big(m,m^{\ell,+}\big)  \left[u(m)-u(m^{\ell,+})\right]^2,
\end{equation}
where $u$ is a test function on $\cD_n$. Put 
\begin{equation}
\label{eq:r_ell}
r_{\ell}=\tilde{r}_n\big(m,m^{\ell,+}\big)=\bar{r}_n\big(\mathbf{t}_n,\mathbf{t}_n^{\ell,+}\big).
\end{equation}
Using \eqref{eq:defE} and \eqref{eq:brn}, we get
\begin{equation}
\label{eq:value_r_ell}
r_{\ell}={\cAl}\,\frac{1-\mathbf{t}_{\ell,n}}{2}\,\exp\left[-2\beta \left(-h -a_{\ell} \left(\frac{a_{\ell}}{n} 
+\sum_{\ell' \in [k]} a_{\ell'} \omega_{\ell',n} \mathbf{t}_{\ell',n} \right)\right)_+\right].
\end{equation}


\subsubsection{Approximation estimates}

Next we estimate how close the pairs $(\bar{r}_n,\tilde{r}_n)$ and $(\mathcal{Q}_n,\tilde{\mathcal{Q}}_n)$ are. By Taylor expansion around $\mathbf{t}_n$, we have
\begin{equation}
\label{eq:Fm-Fm*}
F_n(m)-F_n(\mathbf{t}_n) = \tfrac{1}{2} \Big\langle[m-\mathbf{t}_n], \mathbb{A}_n \cdot [m-\mathbf{t}_n]\Big\rangle + O\big(d(m,\mathbf{t}_n)^3\big).
\end{equation}
In particular, 
\begin{equation}
\label{eq:*F-F}
\begin{split}
&F_n(\mathbf{t}_n^{\ell,\pm})-F_n(\mathbf{t}_n)=\frac{1}{2} \frac{4}{\cAl^2} (\mathbb{A}_n)_{\ell, {\ell}} 
+ O\big(\cAl^{-3}\big)\\
&= \frac{2}{n^2 \omega_{\ell,n}^2} \left[\frac{\omega_{\ell,n}}{\beta}\frac{1}{1-\mathbf{t}_{\ell,n}^2} 
-a_{\ell}^2 \, \omega_{\ell,n}^2 + o\left((n\,\oln)^{-1}\right)\right] + O\big((n\,\oln)^{-3}\big)\\
&= \frac{2}{n^2 } \left(\frac{1}{\beta\,\omega_{\ell,n} (1-\mathbf{t}_{\ell,n}^2)} -a_{\ell}^2 \right) 
+ O\big((n\,\oln)^{-3}\big),
\end{split}
\end{equation}
where the second equality uses \eqref{eq:Anm}. Moreover, for $m \in \cD_n$ ($\ee_{\ell}$ is the unitary vector in $\R^k$ whose $\ell$-th component is non-zero),
\begin{equation}
\label{eq:F-F}
\begin{split}
&F_n(m^{\ell,\pm})-F_n(m)\\
&=\left\langle\left[\pm \tfrac{2}{\cAl}\, \ee_{\ell}\right], \mathbb{A}_n \cdot [m-\mathbf{t}_n]\right\rangle  
+ \frac{1}{2} \left\langle\left[\pm \tfrac{2}{\cAl}\,\ee_{\ell}\right], \mathbb{A}_n  \cdot
\left[\pm \tfrac{2}{\cAl}\,\ee_{\ell}\right]\right\rangle 
+ O\big(d(m,\mathbf{t}_n)^3\big)\\
&=\pm\frac{2}{\cAl} \sum_{\ell' \in [k]}(\mathbb{A}_n)_{\ell, {\ell'}} (m_{\ell'}-\mathbf{t}_{\ell',n}) 
+\frac{2}{\cAl^2} (\mathbb{A}_n)_{\ell, {\ell}}
+ O\big(d(m,\mathbf{t}_n)^3\big)\\
&= \left(\pm\frac{2}{n \omega_{\ell,n}} (m_{\ell}-\mathbf{t}_{\ell,n}) + \frac{2}{n^2 \omega_{\ell,n}^2}\right) 
\left(\frac{\omega_{\ell,n}}{\beta}\frac{1}{1-\mathbf{t}_{\ell,n}^2} -a_{\ell}^2\, 
\omega_{\ell,n}^2 + o(n^{-1})\right) \\
&\quad  \pm\frac{2}{n \omega_{\ell,n}}\sum_{\ell' \in [k],\,\ell'\neq \ell}(-a_{\ell} \, \omega_{\ell,n} \,
a_{{\ell'}} \,\omega_{{\ell'},n})(m_{\ell'}-\mathbf{t}_{\ell',n}) + O\big(n^{-3/2}\log^{3/2} n\big)\\
&=\mp\frac{2}{n}\sum_{\ell' \in [k]} a_{\ell} \,a_{{\ell'}} \,\omega_{{\ell'},n} (m_{\ell'}-\mathbf{t}_{\ell',n})
\pm \frac{2(m_{\ell}-\mathbf{t}_{\ell,n})}{\beta \, n(1-\mathbf{t}_{\ell,n}^2)} + O\big(n^{-3/2}\log^{3/2} n\big),
\end{split}
\end{equation}
where the third equality uses \eqref{eq:Anm}. For $m \in \cD_n$, we have $ d(m,\mathbf{t}_n)^3 = O(n^{-3/2}\log^{3/2}n)$. Therefore, combining \eqref{eq:Qm}, \eqref{eq:tiQ} and \eqref{eq:Fm-Fm*}, we have 
\begin{equation}
\label{eq:QtiQ}
\abs{ \frac{\mathcal{Q}_n(m)}{\tilde{\mathcal{Q}}_n(m)}-1} \leq C'' n^{-1/2}\log^{3/2} n,
\qquad m \in \cD_n, 
\end{equation}
for some $C'' \in (0,\infty)$ constant. Using  \eqref{eq:brn} and \eqref{eq:nE-nE}, we can write 
\begin{equation}
\label{eq:barr}
\begin{aligned}
\bar{r}_n(m,m^{\ell,\pm}) 
&= \exp{\left[- \beta \left[n\left[F_n(m^{\ell,\pm})-F_n(m)\right] -\frac{1}{\beta} \Delta_{\ell,n}^\pm  
+ O\left(n^{-1} \right)\right]_+\right]} \frac{1\mp m_{\ell}}{2},
\end{aligned}
\end{equation}
where $\Delta_{\ell,n}^\pm$ is defined in \eqref{eq:defDelta}. 

Using \eqref{eq:trn}, \eqref{eq:*F-F}, \eqref{eq:F-F} and \eqref{eq:barr}, we find that, for all $m \in \cD_n$, 
\begin{equation}
\label{eq:r/tir}
\begin{aligned}
& \abs{ \frac{\bar{r}_n\left(m,m^{\ell,+} \right)}{\tilde{r}_n\left(m,m^{\ell,+} \right)}-1} 
=  \abs{ \frac{\bar{r}_n\left(m,m^{\ell,+} \right)}{\bar r_n\big(\mathbf{t}_n,\mathbf{t}_n^{\ell,+}\big) }-1}\\
&= \abs{\frac{(1-m_{\ell})\,\exp{\left\{-\left[I_1 + O(n^{-1/2}\log^{3/2} n) 
- \Delta_{\ell,n}^\pm + o_n(1) \right]_+\right\}}
}{(1-\mathbf{t}_{\ell,n})\,\exp{\left\{-\left[I_2 + O(n^{-2}\, \oln^{-3}) 
- \Delta_{\ell,n}^\pm + o_n(1) \right]_+\right\}}} - 1}\\
&= \abs{\frac{(1-m_{\ell})\,\exp{\left\{-\left[ I_1 - \Delta_{\ell,n}^\pm  + o_n(1) \right]_+\right\}}}{
(1-\mathbf{t}_{\ell,n})\,\exp{\left\{-\left[-\Delta_{\ell,n}^\pm  + o_n(1) \right]_+\right\}}} - 1 }
\leq C''' n^{-1/2}\log^{1/2} n,
\end{aligned}
\end{equation}
where $C''' \in (0,\infty)$ is a constant and we abbreviate
\begin{equation}
\begin{aligned}
I_1 &= -2\beta\sum_{\ell' \in [k]} a_{\ell} \,a_{{\ell'}} \,\omega_{{\ell'},n} (m_{\ell'}-\mathbf{t}_{\ell',n}) 
+ \frac{2(m_{\ell}-\mathbf{t}_{\ell,n})}{1-\mathbf{t}_{\ell,n}^2},\\
I_2 &= \frac{2}{n} \left(\frac{1}{\omega_{\ell,n} (1-\mathbf{t}_{\ell,n}^2)} - \beta \, a_{\ell}^2 \right).
\end{aligned}
\end{equation}
Equations \eqref{eq:QtiQ} and \eqref{eq:r/tir} are relevant for the following approximation. 


\subsection{Approximate harmonic function}
\label{sec:appharm}

Let $\mathbb{B}_n$ be the $k \times k$ matrix defined by 
\begin{equation}
\label{eq:defB}
(\mathbb{B}_n )_{\ell \ell'} = \frac{\sqrt{r_\ell r_{\ell'}}}{n \,\omega_{\ell,n}\omega_{\ell',n}}\,
(\mathbb{A}_n)_{\ell \ell'},
\end{equation}
where $\mathbb{A}_n$ is defined in \eqref{eq:An}. Note that 
\begin{equation}
\det \mathbb{B}_n = (\det \mathbb{A}_n)\, \prod_{\ell \in [k]} \frac{r_{\ell}}{n \,\oln^2}.
\end{equation} 
Let $\gamma_n^{(\ell)}$, $\ell \in [k]$, be the eigenvalues of $\mathbb{B}_n$, ordered in increasing order.  Let $\gamma_n=\gamma_n^{(1)}$ denote the unique negative eigenvalue of $\mathbb{B}_n$, and $\hat{v}$ the corresponding unitary eigenvector. Define $v=(v_{\ell})_{\ell \in [k]}$ by  $v_{\ell}=\hat{v}_{\ell}\frac{\oln \sqrt{n}}{\sqrt{r_{\ell}}}$.

\begin{remark}
\label{rem:pos-neg} 
{\rm As in \cite[Remark 10.4]{BdH15}, it follows by Hypothesis~\ref{hyp} that $\mathbb{A}_n$ has all strictly positive eigenvalues but one strictly negative. It can be seen that the same property holds for the eigenvalues of $\mathbb{B}_n$.
}\hfill$\spadesuit$
\end{remark}

\begin{lemma}[{\bf Eigenvalue}]
\label{lem:gamma}
The eigenvalue $\gamma_n$ is the unique solution of the equation 
\begin{equation}
\label{eq:gamma}
\frac{1}{n}\sum_{\ell \in [k]}\frac{a_{\ell}^2}{\frac{1}{n \beta \, \oln (1-\mathbf{t}_{\ell,n}^2)} -\frac{\gamma_n}{r_{\ell}}} 
=1 + O(n^{-1}).
\end{equation}
\end{lemma}

\begin{proof}
We follow the line of proof of \cite[Lemma 14.9]{BdH15}, using the last point in Hypothesis~\ref{hyp}. In our case, \cite[Eq.\ (14.7.12)]{BdH15} reads
\begin{equation}
\label{eq:eig}
-\frac{1}{n} a_\ell \sqrt{r_\ell} \sum_{\ell' \in [k]} a_{\ell'} \sqrt{r_{\ell'}} u_{\ell'} 
+ \left(r_\ell \frac{1}{n\beta\omega_{\ell,n} (1-\mathbf{t}_{\ell,n})^2}-\gamma_n\right) u_\ell + O(n^{-1})
= 0, \qquad \ell \in [k].
\end{equation}
\end{proof}

\begin{remark}
\label{rem:gamma_cond}
{\rm As in \cite[Lemma 14.9]{BdH15}, since the left-hand side of \eqref{eq:gamma} is increasing in $\gamma_n$ for $\gamma_n\geq 0$, a negative solution of \eqref{eq:gamma} exists if and only if 
\begin{equation}
\label{eq:condNeg}
\beta \sum_{\ell \in [k]} a_{\ell}^2 \, \oln (1-\mathbf{t}_{\ell,n}^2)>1.
\end{equation}
Using \eqref{eq:detAm}, \eqref{eq:condNeg} holds if and only if $\det \mathbb{A}_n<0$. By Remark~\ref{rem:pos-neg} the latter holds true.
} \hfill$\spadesuit$
\end{remark}

Define  $f\colon\,\R \to [0,1]$ as
\begin{equation}
\label{eq:def_f}
f(x)=\sqrt{\frac{(-\gamma_n) \beta n}{2 \pi}} \int_{-\infty}^{x} \ee^{-\tfrac{1}{2}(-\gamma_n)\beta n u^2} \dd u
\end{equation}
and $g\colon\,\R^k \to [0,1]$ as
\begin{equation}
\label{eq:def_g}
g(m)=f(\vprod{v}{m-\mathbf{t}_n}).
\end{equation}
Recall the definition of $\mathcal{M}_n(\mathbf{m}_n)$ given in \eqref{eq:defMn}.

Let $W_0$ be a strip in $\Gamma_n$ of width $Cn^{-1/2} \log^{1/2} n$ such that $\mathbf{t}_n \in W_0$, $\mathcal{M}_n(\mathbf{m}_n) \cap W_0$ is empty and $W_0^c$ consists in two non-neighbouring parts: $W_1$ containing $\mathbf{m}_n$ and  $W_2$ containing  $\mathcal{M}_n(\mathbf{m}_n)$. Moreover, we require that, for some fixed constant $c>1$, $W_0 \cap \cD_n^c \subseteq \{m \in \Gamma_n \colon  F_n(m)> F_n(\mathbf{t}_n) + c n^{-1} \log n \}$. Define
\begin{equation}
\label{eq:tilde_g}
\tilde{g}(m) = 
\begin{cases}
0, &m \in W_1,\\
1, &m \in W_2,\\
g(x), &m \in W_0 \cap \cD_n,\\
0, &m \in W_0 \cap \cD_n^c.
\end{cases}
\end{equation}
By choosing $W_0$ and $\cD_n$ suitably we have, for $m \sim m'$ (i.e., $\bar{r}_n(m,m')>0$) and $c \in (0,\infty)$ large enough (coming from the definition of $W_0$), 
\begin{align}
\label{eq:W0Dc}
&\cQ_n(m)\leq \cQ_n(\mathbf{t}_n) n^{-c \beta}, 
&&m \in  W_0 \cap \cD_n^c,\\
\label{eq:W0DW12}
&(\tilde{g}(m)-\tilde{g}(m'))^2 \, \bar{r}_n(m,m') \cQ_n(m) \leq \cQ_n(\mathbf{t}_n) n^{-c \beta}, 
&&m \in  W_0 \cap \cD_n, m' \in W_0^c.
\end{align}


\subsection{Computation of the approximate Dirichlet form}
\label{sec:appDiri}

In this section we follow \cite[Sections 10.2.2--10.2.3]{BdH15} to approximate $\tilde{\cE}_{\cD_n}(g,g)$ defined in \eqref{eq:ti_cE}. As in \cite[Eq. (10.2.24)]{BdH15}, for $m \in \mathcal{D}_n$ and $\ell \in [k]$ such that $m^{\ell,+} \in D_n$, compute
\begin{equation}
\label{eq:g-g}
\begin{split}
&g(m^{\ell,+})-g(m)\\
&\quad = \frac{2}{\cAl} v_{\ell} f'(\vprod{v}{m-\mathbf{t}_n}) +  \frac{2}{\cAl^2} v_{\ell}^2 f''(\vprod{v}{m-\mathbf{t}_n}) 
+ \frac{4}{3\cAl^3} v_{\ell}^3  f'''(\vprod{v}{\tilde{m}-\mathbf{t}_n})\\
&\quad = v_{\ell}\sqrt{\frac{2(-\gamma_n)\beta}{\pi n\,  \oln^2}} \exp\left(-\frac{\beta n}{2} 
(-\gamma_n)\vprod{v}{m-\mathbf{t}_n}^2\right)\\ 
&\quad \qquad \times \left(1-  \frac{1}{\oln}v_{\ell}(-\gamma_n)\beta \vprod{v}{m-\mathbf{t}_n} 
+ O(\oln^{-2} \, n^{-1}\log n) \right).
\end{split}
\end{equation}
Recalling \eqref{eq:ti_cE}--\eqref{eq:r_ell}, we have 
\begin{equation}
	\label{eq:tiEg}
	\begin{split}
		&\tilde{\cE}_{\cD_n}(g,g) =\sum_{m \in \cD_n} \tilde{\mathcal{Q}}_n(m) 
		\sum_{\ell \in [k]}  r_{\ell}  \left[g(m^{\ell,+}) - g(m)\right]^2\\
		&=\frac{1}{Z_n} \sum_{m \in \cD_n}  \exp \left[-\tfrac{\beta n}{2}\big\langle[m-\mathbf{t}_n],
		\mathbb{A}_n\cdot[m-\mathbf{t}_n]\big\rangle\right] \ee^{-\beta n F_n(\mathbf{t}_n)}\\
		& \quad \times \sum_{\ell \in [k]} r_{\ell} v_{\ell}^2 \,\frac{2 (-\gamma_n) \beta}{\pi n \, \oln^2} 
		\exp\left(-\beta n(-\gamma_n)\vprod{v}{m-\mathbf{t}_n}^2\right)\\ 
		&\qquad \times \left(1-  \frac{v_{\ell}}{\oln}(-\gamma_n) \beta \vprod{v}{m-\mathbf{t}_n} + O(\oln^{-2} \, n^{-1}\log n) \right)^2\\
		&= \frac{1}{Z_n} \frac{2(-\gamma_n)\beta }{\pi} \sum_{m \in \cD_n}  
		\exp \left[-\tfrac{\beta n}{2}\big\langle [m-\mathbf{t}_n],\mathbb{A}_n\cdot[m-\mathbf{t}_n]\big\rangle\right] 
		\ee^{-\beta n F_n(\mathbf{t}_n)}\\
		& \quad \times \exp\left(-\beta n (-\gamma_n)\vprod{v}{m-\mathbf{t}_n}^2\right) 
		\left[1+O\left(\oln^{-1} \,n^{-1/2} \log^{1/2} n\,\right)\right]\\
		&= \frac{1}{Z_n} \frac{2(-\gamma_n)\beta }{\pi} \left[1+O\left(\oln^{-1} \,\,n^{-1/2} \log^{1/2} n\,\right)\right] 
		\ee^{-\beta n F_n(\mathbf{t}_n)} \left(\prod_{\ell \in [k]} \frac{\cAl}{2}\right)  \\
		& \quad \times  \int_{\cD_n} \dd m \exp \left[-\tfrac{\beta n}{2}\big\langle [m-\mathbf{t}_n],
		\mathbb{A}_n\cdot[m-\mathbf{t}_n]\big\rangle\right] \exp\left(-\beta n (-\gamma_n)\vprod{v}{m-\mathbf{t}_n}^2\right)\\
		&= \frac{1}{Z_n} \ee^{-\beta n F_n(\mathbf{t}_n)}  \frac{(-\gamma_n) n}{\sqrt{[-\det \mathbb{A}_n]}} 
		\left(\frac{\pi n}{2 \beta}\right)^{\frac{k}{2}-1} \left(\prod_{\ell \in [k]}{\oln}\right) 
		\left[1+O\left(\oln^{-1} \,n^{-1/2} \log^{1/2} n\,\right)\right],
	\end{split}
\end{equation}
where we use \cite[Eq. (10.2.33)]{BdH15} with $\varepsilon=\frac{1}{\beta n}$ and $d=k$. Here $\tfrac12 \cAl$ is the inverse of the step in the $\ell$--direction, while in \cite[Eq. (10.2.33)]{BdH15} the step is $\varepsilon$.

\begin{remark}
\label{rem:g_ti_g}
{\rm Note that  
\begin{equation}
\label{eq:g_ti_g}
\tilde{\cE}_{\cD_n}(g,g)=\tilde{\cE}_{\cD_n}(\tilde{g},\tilde{g})\,[1+o(1)]
\end{equation} 
because $\tilde{g}(m)=g(m)\,[1+o(1)]$ for all $m \in W_0^c \cap \cD_n$. The latter can be proved by approximating the Gaussian integral by $0$ or $1$ when $\vprod{v}{ m-\mathbf{t}_n}$ is proportional to $- n^{-1/2} \log^{1/2} n$ or $n^{-1/2} \log^{1/2} n$, respectively.}\hfill$\spadesuit$ 
\end{remark}


\subsection{Final Dirichlet form approximation}
\label{sec:finalDiri}

We are now ready to compare ${\cE}_{\cS_n} $ with $\tilde{\cE}_{\cD_n}$. Let $h\colon \cS_n \to [0,1]$ be such that $h(\si)=\tilde g(m_n(\si))$, $\si \in \cS_n$. We split the sum in \eqref{eq:Dirialt} into four subsets of $\Gamma_n\times \Gamma_n$:  $m \in  W_0 \cap \cD_n^c$, $m' \in \Gamma_n$; $m \in  W_0 \cap \cD_n$, $m'\in W_1$; $m \in  W_0 \cap \cD_n$, $m'\in W_2$; $m \in  W_0 \cap \cD_n$, $m'\in W_0 \cap \cD_n$. Then, using \eqref{eq:tilde_g}--\eqref{eq:W0DW12}, we obtain
\begin{equation}
\label{eq:appr1}
\begin{split}
{\cE}_{\cS_n}(h,h)  
&= O(n^{-c \beta}) +\frac{1}{2} \sum_{m \in  W_0 \cap \cD_n} 
\sum_{m' \in W_0 \cap \cD_n} \mathcal{Q}_n(m) \,\bar r_n\big(m, m'\big)\,\big[\,\tilde{g}(m)-\tilde{g}(m')\big]^2.
\end{split}
\end{equation}
Using \eqref{eq:QtiQ} and \eqref{eq:r/tir}, we obtain
\begin{equation}
\label{eq:testDir}
\begin{split}
{\cE}_{\cS_n}(h,h)  
&= O(n^{-c \beta}) +\frac{1}{2} \sum_{m \in  W_0 \cap \cD_n} 
\left[1+O\big(n^{-1/2} \log^{3/2} n\big) \right] \tilde{\mathcal{Q}}_n(m)\\
& \quad \times \sum_{m' \in W_0 \cap \cD_n} \, \left(1+O\big(n^{-1/2} \log^{1/2} n\big) \right) 
\tilde{r}_n\big(m, m'\big)\,\big[\,\tilde{g}(m)-\tilde{g}(m')\big]^2\\
&=  \left[1+O\big(n^{-1/2} \log^{1/2} n\big)\right] \frac{1}{2}\sum_{m, m' \in W_0 \cap \cD_n} 
\tilde{\mathcal{Q}}_n(m) \tilde{r}_n\big(m, m'\big)\,\big[\,\tilde{g}(m)-\tilde{g}(m')\big]^2\\
&= \left[1+O\big(n^{-1/2} \log^{1/2} n\big) \right] \frac{1}{2}\sum_{m, m' \in \cD_n} 
\tilde{\mathcal{Q}}_n(m) \tilde{r}_n\big(m, m'\big)\,\big[\,\tilde{g}(m)-\tilde{g}(m')\big]^2\\
&= \tilde{\cE}_{\cD_n}(\tilde{g},\tilde{g}) \left[1+O\big(n^{-1/2} \log^{1/2} n\big) \right]\\
&= [1+o_n(1)]\, \frac{1}{Z_n} \exp\left[-\beta n F_n(\mathbf{t}_n)\right]  
\frac{(-\gamma_n) n}{\sqrt{[-\det \mathbb{A}_n]}} 
\left(\frac{\pi n}{2 \beta}\right)^{\frac{k}{2}-1}  \left(\prod_{\ell \in [k]}{\oln}\right),
\end{split}
\end{equation}
where the third equality follows from \eqref{eq:tilde_g}--\eqref{eq:W0DW12} together with \eqref{eq:QtiQ}, and the last equality follows from \eqref{eq:tiEg}--\eqref{eq:g_ti_g}. 


\section{Capacity and valley estimates}
\label{sec:cap-val}

Section~\ref{sec:capharm} provides sharp asymptotic upper bounds and lower bounds on the capacity of the metastable pair between which the crossover is being considered. These estimates use the results of the Section~\ref{sec:Dir} together with the Dirichlet principle and the Berman-Konsowa principle, which are variational representations of capacity. Section~\ref{sec:valley} provides a sharp asymptotic estimate for the mesoscopic measure of the valleys of the minima of $F_n$, which leads to a sharp asymptotic estimate for $F_n$ inside this valley.

\subsection{Capacity estimates}
\label{sec:capharm}

Given a Markov process $(x_t)_{t\geq 0}$ with state space $S$, a key quantity in the potential-theoretic approach to metastability is the \emph{capacity} $\capa(A,B)$ of two disjoint subsets $A,B$ of $S$. This is defined by (see \cite[Eq. (7.1.39)]{BdH15})
\begin{equation}
\capa(A,B)= \sum_{x \in A} \mu(x) \mathbb{P}_{x}(\tau_B<\tau_A),
\end{equation}
where $\mu$ is the invariant measure and $\mathbb{P}_{x}$ is the probability distribution of the Markov process starting in $x$.

Recall that $\mathcal{M}_n$ is the set of local minima of $F_n$.

\begin{proposition}[{\bf Asymptotics of the capacity}]
\label{prop:cap}
Let $\mathbf{m}_n=(\mathbf{m}_{\ell,n})_{\ell \in [k]} \in  \mathcal{M}_n$  and $M_n \subset \mathcal{M}_n \backslash \mathbf{m}_n$, such that the gate $\mathcal{G}(\mathbf{m}_n,M_n)$ consists of a unique point $\mathbf{t}_n=(\mathbf{t}_{\ell,n})_{\ell \in [k]}$. Suppose that $\beta \in (\beta_c,\infty)$ and $h \in [0,h_c(\beta))$. Then, as $n\to\infty$,
\begin{equation}
\begin{aligned}
&\capa(\cS_n[\mathbf{m}_n], \cS_n[M_n])\\
&\quad = [1+o_n(1)]\,\frac{1}{Z_n} \ee^{-\beta n F_n(\mathbf{t}_n)} 
\frac{(-\gamma_n)\,n}{\sqrt{[-\det (\mathbb{A}_n(\mathbf{t}_n))]}} \left(\frac{\pi n}{2 \beta}\right)^{\frac{k}{2}-1} 
\left(\prod_{\ell \in [k]}{\oln}\right).
\end{aligned}
\end{equation}
\end{proposition}

\begin{remark}
{\rm Proposition~\ref{prop:cap} holds for any subset $M_n \subseteq \mathcal{M}_n \backslash \mathbf{m}_n$, separated from $\mathbf{m}_n$ by $\mathbf{t}_n$, independently on the values of $F_n$ on $M_n$.} \hfill$\spadesuit$
\end{remark}


\subsubsection{Upper bound: Dirichlet principle}

An important characterisation of the capacity between two disjoint sets is given by the \emph{Dirichlet principle}. For our quantity of interest this states that
\begin{equation}
\label{eq:DirPrinc}
\capa(\cS_n[\mathbf{m}_n], \cS_n[M_n])= \inf_{u \in \tilde{\mathcal{H}}} {\cE}_{\cS_n}(u,u),  
\end{equation}
where $\tilde{\mathcal{H}}$ is the set of functions from $\cS_n$ to $[0,1]$ that are equal to $1$ on $\cS_n[\mathbf{m}_n]$ and $0$ on $\cS_n[M_n]$.

Given that, by assumption, $\mathcal{G}(\mathbf{m}_n,M_n)=\{\mathbf{t}_n\}$, we use the Dirichlet principle in \eqref{eq:DirPrinc} to obtain an upper bound on the capacity. We take as test function $h \in \tilde{\mathcal{H}}$ defined in Section~\ref{sec:finalDiri} and, using \eqref{eq:testDir}, we obtain
\begin{equation}
\label{eq:cap_upper}
\begin{split}
&\capa(\cS_n[\mathbf{m}_n], \cS_n[M_n]) \leq  {\cE}_{\cS_n}(h,h)\\
&= [1+o_n(1)]\, \frac{1}{Z_n} \ee^{-\beta n F_n(\mathbf{t}_n)}  \frac{(-\gamma_n)n}{\sqrt{[-\det (\mathbb{A}_n(\mathbf{t}_n))]}} 
\left(\frac{\pi n}{2 \beta}\right)^{\frac{k}{2}-1} \left(\prod_{\ell \in [k]}{\oln}\right).
\end{split}
\end{equation}


\subsubsection{Lower bound: Berman-Konsowa principle}

We first note that the process $(\si_t)_{t \geq 0}$ is lumpable. Indeed, the process $(m_n(\si_t))_{t \geq 0}$ is Markovian because the Hamiltonian $H_n(\si)$ depends on $m_n(\si)$ only (see \eqref{eq:H=nE}). Therefore, for $\mathbf{A}=S_n[{A}]$ and $\mathbf{B}=S_n[{B}]$ with $A$ and $B$ disjoint subsets of $\Gamma_n$,
\label{notationGamma}
\begin{equation}
\capa(\mathbf{A},\mathbf{B})=\capa_{\Gamma}(A,B),
\end{equation}
where $\capa_{\Gamma}$ denotes the capacity for the process $(m_n(\si_t))_{t \geq 0}$, i.e., the projection of the process $(\si_t)_{t\geq 0}$ on the magnetisation space $\Gamma_n$. We write $\mathbb{P}^{\Gamma}$  and $\mathbb{E}^{\Gamma}$ to  denote the law of $(m_n(\si_t))_{t \geq 0}$ induced by the law $\mathbb{P}$ of $(\si_t)_{t\geq 0}$, and its expectation, respectively. By the lumpability, we can focus on the dynamics on $\Gamma_n$. 

Following the line of argument in \cite[Section 10.3]{BdH15} (with $\varepsilon=\frac{2}{n}$ and $d=k$), we obtain the lower bound 
\begin{equation}
\label{eq:lower_cap}
\begin{split}
&\capa(\cS_n[\mathbf{m}_n], \cS_n[M_n]) =\capa_{\Gamma}(\mathbf{m}_n,M_n) 
\geq \tilde{\cE}_{\cD_n}(\tilde{g},\tilde{g}) \left[1+O(\,n^{-1/2} \log^{1/2} n)\right]\\
& =  \frac{1}{Z_n} \ee^{-\beta n F_n(\mathbf{t}_n)}  
\frac{(-\gamma_n)n}{\sqrt{[-\det (\mathbb{A}_n(\mathbf{t}_n))]}} \left(\frac{\pi n}{2 \beta}\right)^{\frac{k}{2}-1}  
\left(\prod_{\ell \in [k]}{\oln}\right)  \left[1+o_n(1)\right],
\end{split}
\end{equation}
where we use \eqref{eq:tiEg} and \eqref{eq:g_ti_g}.

We sketch the proof. The main idea is to use the Berman-Konsowa principle for a suitable defective flow. More precisely, given disjoint subsets $A,B$ of the state space, for any \emph{defective loop-free unit flow} $f_{A,B}$ from $A$ to $B$ with defect function $\delta$ (as defined in \cite[Definition 9.2]{BdH15}), we can estimate (see \cite[Lemma 9.4]{BdH15}, and notation therein)
\begin{equation}
\label{eq:BermanKonsova}
\capa(A,B)\geq \prod_{i=1}^M \left(1+\left[\max_{y \in A_{i}}\frac{\delta(y)}{\mathcal{F}(y)}\right]_+ \right)^{-1} \sum_{\gamma}\bP^{f_{A,B}} (\gamma)\left[\left(\sum_{(x,y)\in \gamma} 
\frac{f_{A,B}((x,y))}{\mu(x)p(x,y)}\right)^{-1}\right],
\end{equation}
where $[\cdot]_+$ denotes the positive part and $\gamma$ is a self-avoiding path from $A$ to $B$. It turns out that, with a suitable choice of the flow $f$, the product in the right-hand side of \eqref{eq:BermanKonsova} is bounded from below by $1+O(\,n^{-1/2} \log^{1/2} n) $, and the sum over $\gamma$ from below by $\tilde{\cE}_{\cD_n}(\tilde{g},\tilde{g}) [1+o_n(1)]$. This proves \eqref{eq:lower_cap}. 

We give a sketch of the test flow definition in our setting. Here $A=\{\mathbf{m}_n\}$ and $B=M_n$. Let $v^*$ be the eigenvector corresponding to the unique negative eigenvalue of the Hessian of $F_n$ at the saddle point $\mathbf{t}_n$ (unique gate point in $\mathcal{G}(\{\mathbf{m}_n\},M_n)$). Let $G_n$ be the cylinder in $\R^k$ intersected with $\Gamma_n$, centred at $\mathbf{t}_n$, with axis $v^*$, radius $\rho=C\,n^{-1/2} \log^{1/2} n$ and length $\rho'=C'\,n^{-1/2} \log^{1/2} n$.  We will denote by $\partial_B G_n$ the base facing $B$ and by $\partial_A G_n$ the central part of radius $C''\,n^{-1/2} \log^{1/2} n$ of the base facing $A$, with $C''<C$. Choose the constants so that $G_n$ is contained in $\cD_n$ defined in \eqref{eq:defDn}.

We define a defective flow $f_{A,B}$ from $A$ to $B$ consisting of three parts: $f_A$, a unitary flow from $A$ to $\partial_A G_n$; $f$, a defective loop-free unit flow from $\partial_A G_n$ to $\partial_B G_n$ inside $G_n$; $f_B$, a unitary flow from $\partial_B G_n$ to $B$. This choice implies that the sum over $\gamma$ in  \eqref{eq:BermanKonsova} is relevant only on the paths entering $G_n$ in $\partial_A G_n$, exiting $G_n$ in $\partial_B G_n$, and afterwards reaching $B$ without going back to $G_n$. For this purpose we choose $f_A$ and $f_B$ such that $f_A((x,y))$ and $f_B((x,y))$ are proportional to $\cQ_n(x)$. For $m \in G_n$ such that  $m^{\ell,+} \in G_n$, define 
\begin{equation}
f((m, m^{\ell,+})) = \frac{\tilde{\cQ}_n(m) r_{\ell} \left[g(m^{\ell,+}) - g(m)\right]_+ }{N(g)},
\end{equation}
where $g$ is defined in \eqref{eq:def_g}, $\tilde{\cQ}_n$ in \eqref{eq:tiQ}, $r_{\ell}$ in \eqref{eq:r_ell} and 
\begin{equation}
N(g) = \sum_{m \in \partial_A G_n} \sum_{\substack{\ell \in [k]: \\ m^{\ell,+} \in G_n }} 
\tilde{\cQ}_n(m) r_{\ell} \left[g(m^{\ell,+}) - g(m)\right]_+.
\end{equation}
The contribution to the sum in brackets in \eqref{eq:BermanKonsova} turns out to be negligible outside $G_n$. Therefore, no further conditions on the flows $f_A$ and $f_B$ are necessary, provided the total flow out of $A$ is $1$ and the total flow $f_{A,B}$ is defective and loop-free.


\subsection{Measure of the valley}
\label{sec:valley}

In order to prove Theorem~\ref{thm:metER}, we need the following estimate on the measure of the valley of the minima of $F_n$. For $\mathbf{m}_n \in \mathcal{M}_n$, let $A(\mathbf{m}_n)\subset \Gamma_n$ be the valley of $\mathbf{m}_n$ as defined in \cite[Eq. (8.2.10)]{BdH15}. 

\begin{lemma}[{\bf Gibbs weight of the valley}]
\label{lem:QAm1}
Given $\mathbf{m}_n \in \mathcal{M}_n$, 
\begin{equation}
\cQ_n(A(\mathbf{m}_n))=\frac{1}{Z_n} \frac{\exp\left(-\beta n F_n(\mathbf{m}_n)\right)}{\sqrt{\det (\mathbb{A}_n(\mathbf{m}_n))}} 
\left(\frac{n \pi}{2\beta}\right)^{\frac{k}{2}} \left(\prod_{\ell \in [k]}{\oln}\right)  
\left[1+O(\,n^{-1/2} \log^{3/2} n)\right],
\end{equation}
where $\cQ_n$ is the mesoscopic measure defined in \eqref{eq:Qm},  and $\mathbb{A}_n(\mathbf{m}_n)$ is the $k \times k$ Hessian matrix defined in \eqref{eq:defAnm}.
\end{lemma}

\begin{proof}
The proof follows that of \cite[Lemma 10.12 and (10.2.33)]{BdH15}. The relevant contribution to $\cQ_n(A(\mathbf{m}_n))$ is given by the measure of a ball $B_{\rho}$ of radius $\rho=C\,n^{-1/2} \log^{1/2} n$ centred in $\mathbf{m}_n$, with $C$ constant, contained in $A(\mathbf{m}_n)$. Indeed, if $y \in A(\mathbf{m}_n)$ and $d(\mathbf{m}_n,y)> \rho$, then by Taylor expansion of $F_n$ around $\mathbf{m}_n$ we have 
\begin{equation}
\begin{split}
\cQ_n(y)
&= \frac{1}{Z_n}\exp[-\beta n F_n(y)]
=\frac{1}{Z_n}\exp\left[-\beta n [F_n(\mathbf{m}_n)+c \,d(\mathbf{m}_n,y)^2]\right] \\
& \leq \frac{1}{Z_n}\exp\left[-\beta n [F_n(\mathbf{m}_n)+c \rho^2 ]\right] 
= \frac{n^{-\beta c C^2}}{Z_n}\exp\left[-\beta n F_n(\mathbf{m}_n)\right],
\end{split}
\end{equation}
where $c$ is a constant. The condition $y \in A(\mathbf{m}_n)$ is needed to ensure that $F_n(y)>F_n(\mathbf{m}_n)$, implying that $c$ is positive. Therefore, we obtain the rough estimate
\begin{equation}
\label{eq:minusB}
\cQ_n(A(\mathbf{m}_n)\backslash B_{\rho}) 
\leq n^{k}\frac{n^{-\beta c C^2}}{Z_n} \exp\left[-\beta n F_n(\mathbf{m}_n)\right],
\end{equation} 
where we use that $|\Gamma_n| \leq n^k$. The bound in \eqref{eq:minusB} is sufficient to show that $\cQ_n(A(\mathbf{m}_n)\backslash B_{\rho}) $ is negligible in $\cQ_n(A(\mathbf{m}_n))$.

Compute
\begin{equation}
\label{eq:QB}
\begin{split}
&Z_n \cQ_n(A(\mathbf{m}_n) \cap B_{\rho})= Z_n\cQ_n(B_{\rho}) = Z_n \sum_{y \in B_{\rho}} \cQ_n(y)
= \sum_{y \in B_{\rho}} \ee^{-\beta n F_n(y)}\\
&= \ee^{-\beta n F_n(\mathbf{m}_n)} \sum_{y \in B_{\rho}} \exp \left[-\frac{\beta n}{2} 
\vprod{y-\mathbf{m}_n}{(\mathbb{A}_n(\mathbf{m}_n))\cdot(y-\mathbf{m}_n)} + O(n\rho^3)\right]\\
&=  \ee^{-\beta n F_n(\mathbf{m}_n)} [1+O(n\rho^3)]\sum_{y \in B_{\rho}} \exp \left[-\frac{\beta n}{2} 
\vprod{y-\mathbf{m}_n}{(\mathbb{A}_n(\mathbf{m}_n))\cdot(y-\mathbf{m}_n)} \right]\\
&=  \ee^{- \beta n F_n(\mathbf{m}_n)} \left(\prod_{\ell \in [k]} \frac{\cAl}{2}\right)[1+O(n\rho^3)]\\
& \times \int_{ B_{\rho}} \dd y 
\exp \left[- \frac{\beta n}{2} \vprod{y-\mathbf{m}_n}{(\mathbb{A}_n(\mathbf{m}_n))\cdot(y-\mathbf{m}_n)} \right]\\
&= \ee^{-\beta n F_n(\mathbf{m}_n)} \left(\frac{n}{2}\right)^k  \left(\prod_{\ell \in [k]}{\oln}\right) 
[1+O(n\rho^3)]\left(\frac{2\pi}{n \beta}\right)^{\frac{k}{2}} \sqrt{\frac{1}{\det (\mathbb{A}_n(\mathbf{m}_n))}}\\
&= \frac{\ee^{-\beta n F_n(\mathbf{m}_n)}}{\sqrt{\det (\mathbb{A}_n(\mathbf{m}_n))}} 
\left(\frac{n \pi}{2\beta}\right)^{\frac{k}{2}} \left(\prod_{\ell \in [k]}{\oln}\right)  [1+O(n\rho^3)],
\end{split}
\end{equation}
where we use the Taylor expansion 
\begin{equation}
F_n(y) = F_n(\mathbf{m}_n) + \frac{1}{2} \vprod{y-\mathbf{m}_n}{(\nabla^2 F_n)(\mathbf{m}_n)\cdot(y-\mathbf{m}_n)} 
+ O(\rho^3), \qquad y \in B_{\rho},
\end{equation} 
and the approximation of the sum by an integral is correct up to an error $1+O(\rho)$. In the last lines we approximated the Gaussian integral on intervals $[-\rho, \rho]$ by the Gaussian integral on $\R$, with an error $1+O(n^{-c})$. We conclude by looking at \eqref{eq:minusB} and \eqref{eq:QB}, and noting that for $C$ large enough $\cQ_n(A(\mathbf{m}_n)\backslash B_{\rho})$ is negligible compared to $ \cQ_n(A(\mathbf{m}_n) \cap B_{\rho})$.
\end{proof}



\section{Proof of the theorems}
\label{sec:proofs}

In this section we prove Theorems~\ref{thm:metER}--\ref{thm:metER_lim}. Section~\ref{sec:avecross} uses the asymptotics for the capacity of the metastable pair from Section~\ref{sec:capharm} and the asymptotics for the mesoscopic measure from Section~\ref{sec:valley} to prove Theorem~\ref{thm:metER}. Section~\ref{sec:exp} proves Theorem~\ref{thm:exp_law}. Section~\ref{sec:correction} proves Theorem~\ref{thm:metER_lim}.


\subsection{Average crossover time}
\label{sec:avecross}

Let us return to the notation of Theorem~\ref{thm:metER}, where $\mathbf{m}_n \in \mathcal{M}_n$ and $\mathcal{M}_n(\mathbf{m}_n)=\{m \in \mathcal{M}_n \backslash \mathbf{m}_n \colon\, F_n(m) \leq F_n(\mathbf{m}_n)\}$. To prove Theorem~\ref{thm:metER} we use the relation 
\begin{equation}
\label{eq:mu/cap}
\mathbb{E}^{\Gamma}_{\mathbf{m}_n}(\tau_{\mathcal{M}_n(\mathbf{m}_n)})
=  [1+o_n(1)]\,\frac{\mu(A(\mathbf{m}_n))}{\capa_{\Gamma}(\mathbf{m}_n,\mathcal{M}_n(\mathbf{m}_n))},
\end{equation}
Recall notation introduced in Section~\ref{notationGamma}.
Because $F_n(m) \leq F_n(\mathbf{m}_n)$ for all $m \in \mathcal{M}_n(\mathbf{m}_n)$, \eqref{eq:mu/cap} follows from \cite[Theorem 8.15]{BdH15} after proving that $\mathcal{M}_n$  is a set of metastable points in the sense of \cite[Definition 8.2]{BdH15}. The latter follows along the lines of the proof of \cite[Theorem 10.6]{BdH15}, where similar values of capacities and invariant measures occur.

Using \eqref{eq:mu/cap} in combination with Proposition~\ref{prop:cap} and Lemma~\ref{lem:QAm1}, we obtain that, for all $\si \in \cS_n[\mathbf{m}_n]$,
\begin{equation}
\begin{split}
\mathbb{E}_{\si}(\tau_{\cS_n[\mathcal{M}_n(\mathbf{m}_n)]})
&= \mathbb{E}^{\Gamma}_{\mathbf{m}_n}(\tau_{\mathcal{M}_n(\mathbf{m}_n)}) 
=[1+o_n(1)]\,\frac{\mathcal{Q}_n(A(\mathbf{m}_n))}{\capa_{\Gamma}(\mathbf{m}_n, \mathcal{M}_n(\mathbf{m}_n))} \\
&= [1+o_n(1)]\,\frac{\mathcal{Q}_n(A(\mathbf{m}_n))}{\capa(\cS_n[\mathbf{m}_n], 
\cS_n[\mathcal{M}_n(\mathbf{m}_n)])}\\
&= [1+o_n(1)]\,\frac{\frac{1}{Z_n} \frac{\exp\left(-\beta n F_n(\mathbf{m}_n)\right)}
{\sqrt{\det (\mathbb{A}_n(\mathbf{m}_n))}} 
\left(\frac{n \pi}{2\beta}\right)^{\frac{k}{2}} \left(\prod_{\ell \in [k]}{\oln}\right)}
{\frac{1}{Z_n}\exp\left[-\beta n F_n(\mathbf{t}_n)\right]  \frac{(-\gamma_n)n}{\sqrt{[-\det (\mathbb{A}_n(\mathbf{t}_n))]}} 
\left(\frac{\pi n}{2 \beta}\right)^{\frac{k}{2}-1} \left(\prod_{\ell \in [k]}{\oln}\right)}\\
&= [1+o_n(1)]\,\sqrt{\frac{[-\det (\mathbb{A}_n(\mathbf{t}_n))]}{\det (\mathbb{A}_n(\mathbf{m}_n))}}
\left(\frac{\pi }{2 \beta(-\gamma_n)}\right) \exp\left[\beta n (F_n(\mathbf{t}_n)-F_n(\mathbf{m}_n))\right],
\end{split}
\end{equation}
where we use that the dynamics depends on the starting configuration $\si\in \cS_n[\mathbf{m}_n]$ only, through its level magnetisations $m_n(\si)=\mathbf{m}_n$ (see \eqref{eq:H=nE}), and also use the lumpability. 


\subsection{Exponential law}
\label{sec:exp}

In this section we prove Theorem~\ref{thm:exp_law}. Since the dynamics depends on the starting configuration $\si\in \cS_n[\mathbf{m}_n]$ through its level magnetisation $m_n(\si)=\mathbf{m}_n$ only (see \eqref{eq:H=nE}), we have
\begin{equation}
\label{eq:lamp_expLaw}
\begin{split}
&\lim_{n\to\infty} \mathbb{P}_{\si} \left( \tau_{\cS_n[\mathcal{M}_n(\mathbf{m}_n)]}> t \, 
\mathbb{E}_\sigma\left[\tau_{\cS_n[\mathcal{M}_n(\mathbf{m}_n)]}\right] \right) 
= \lim_{n\to\infty} \mathbb{P}^{\Gamma}_{\mathbf{m}_n} \left( \bar{\tau}_{\mathcal{M}_n(\mathbf{m}_n)}> t \, 
\mathbb{E}^{\Gamma}_{\mathbf{m}_n} \left[\bar{\tau}_{\mathcal{M}_n(\mathbf{m}_n)}\right] \right),
\end{split}
\end{equation}
where $\bar{\tau}$ is the hitting time of the process projected on $\Gamma_n$. Given the non-degeneracy hypothesis (Hypothesis~\ref{hyp} in Section~\ref{sec:crossover}) and the one-dimensional landscape analysis (in Section~\ref{sec:1dim}), we can apply \cite[Theorem  8.45]{BdH15} to the right-hand side of \eqref{eq:lamp_expLaw} and conclude the proof. 


\subsection{Randomness of the exponent} 
\label{sec:correction}

In this section we prove Theorem~\ref{thm:metER_lim}. In particular, we compute $F_n(\mathbf{t}_n)-F_n(\mathbf{m}_n) - [F_{\beta,h}(\mathbf{t})-F_{\beta,h}(\mathbf{m})]$ to leading order.

Recalling definitions \eqref{eq:F} and \eqref{eq:defKm}, 
we have
\begin{equation}
F_{\beta,h}(m)
=  -\frac12  K(m)^2  -h \sum_{\ell \in [k]} \, \omega_{\ell} \, m_{\ell} 
+\frac{1}{\beta}  \sum_{\ell \in [k]}\omega_{\ell} I_{\mathbf{C}}(m_{\ell}).
\end{equation} 

Let $\mathbf{m}=(\mathbf{m}_{\ell})_{\ell \in [k]}, \mathbf{t}=(\mathbf{t}_{\ell})_{\ell \in [k]} \in [-1,1]^k$ be the critical points of $F_{\beta,h}$ closest to $\mathbf{m}_n, \mathbf{t}_n$ (i.e., the critical points of $F_n$ defined above), respectively. 
Note that $\mathbf{m} $ and $\mathbf{t}$ satisfy \eqref{eq:m*}, while $\mathbf{m}_n $ and $\mathbf{t}_n$ satisfy \eqref{eq:crit}. 
Using \eqref{eq:In}, we get
\begin{equation}
\begin{split}
&F_n(\mathbf{t}_n)-F_{\beta,h}(\mathbf{t}_n)
= -\frac{1}{2} [K_n(\mathbf{t}_n)^2-K(\mathbf{t}_n)^2] -h \sum_{\ell \in [k]} [\omega_{\ell,n}-\omega_{\ell}]\, \mathbf{t}_{\ell,n}\\
&\quad + \frac{1}{\beta} \left[\sum_{\ell \in [k]} [\omega_{\ell,n}-\omega_{\ell}] I_{\mathbf{C}}(\mathbf{t}_{\ell,n}) 
+ \sum_{\ell \in [k]} \frac{1}{2n} \log \left(\frac{\pi (1-\mathbf{t}_{\ell,n}^2)}{2}\right) \omega_{\ell,n} -\frac{k}{2n} 
+ o\left(n^{-1}\right)\right]
\end{split}
\end{equation}
and
\begin{equation}
\begin{split}
&F_{\beta,h}(\mathbf{t}_n)-F_{\beta,h}(\mathbf{t})
= -\frac{1}{2} [K(\mathbf{t}_n)^2-K(\mathbf{t})^2] +\frac{1}{\beta} \sum_{\ell \in [k]} \omega_{\ell}
[I_{\mathbf{C}}(\mathbf{t}_{\ell,n})-I_{\mathbf{C}}(\mathbf{t}_{\ell})].
\end{split}
\end{equation}
By \eqref{eq:crit11}, we have 
\begin{equation}
\label{eq:log+-}
\begin{aligned}
\frac{1}{2}\log\left( \frac{1+\mathbf{t}_{\ell,n}}{1-\mathbf{t}_{\ell,n}}\right) 
&= \beta \left[ a_{\ell} K_n(\mathbf{t}_{n}) +h\right],\\
\frac{1}{2}\log\left( \frac{1+\mathbf{t}_{\ell}}{1-\mathbf{t}_{\ell}}\right) 
&= \beta \left[ a_{\ell} K(\mathbf{t}) +h\right].
\end{aligned}
\end{equation}
Thus,
\begin{equation}
\label{eq:Taylor_Ic}
\begin{split}
I_{\mathbf{C}}(\mathbf{t}_{\ell,n})-I_{\mathbf{C}}(\mathbf{t}_{\ell})
&=(\mathbf{t}_{\ell,n}-\mathbf{t}_{\ell}) I_{\mathbf{C}}'(\mathbf{t}_{\ell}) + O((\mathbf{t}_{\ell,n}-\mathbf{t}_{\ell})^2)\\
&= (\mathbf{t}_{\ell,n}-\mathbf{t}_{\ell}) \frac{1}{2}\log \left(\frac{1+\mathbf{t}_{\ell}}{1-\mathbf{t}_{\ell}}\right) 
+ O((\mathbf{t}_{\ell,n}-\mathbf{t}_{\ell})^2)\\
&=  (\mathbf{t}_{\ell,n}-\mathbf{t}_{\ell}) \beta \left[ a_{\ell} K(\mathbf{t}) +h\right] 
+ O((\mathbf{t}_{\ell,n}-\mathbf{t}_{\ell})^2).
\end{split}
\end{equation}
Moreover,
\begin{equation}
\begin{split}
&K(\mathbf{t}_n)^2-K(\mathbf{t})^2= \sum_{\ell,\ell' \in [k]} a_{\ell}\,a_{\ell'} \, \omega_{\ell}\,\omega_{\ell'} 
[\mathbf{t}_{\ell,n}\mathbf{t}_{\ell',n} - \mathbf{t}_{\ell}\mathbf{t}_{\ell'}]\\
&=  \sum_{\ell,\ell' \in [k]} a_{\ell}\,a_{\ell'} \, \omega_{\ell}\,\omega_{\ell'} \left(\mathbf{t}_{\ell} 
[\mathbf{t}_{\ell',n} - \mathbf{t}_{\ell'}] + \mathbf{t}_{\ell'} [\mathbf{t}_{\ell,n} - \mathbf{t}_{\ell}] 
+ [\mathbf{t}_{\ell,n} - \mathbf{t}_{\ell}] [\mathbf{t}_{\ell',n} - \mathbf{t}_{\ell'}]\right)
\end{split}
\end{equation}
and
\begin{equation}
\begin{split}
&K_n(\mathbf{t}_n)^2-K(\mathbf{t}_n)^2= \sum_{\ell,\ell' \in [k]} a_{\ell}\,a_{\ell'} \, 
[\omega_{\ell,n}\,\omega_{\ell',n}-\omega_{\ell}\,\omega_{\ell'}] \mathbf{t}_{\ell,n}\mathbf{t}_{\ell',n}\\
& = \sum_{\ell,\ell' \in [k]} a_{\ell}\,a_{\ell'} \, \mathbf{t}_{\ell,n}\mathbf{t}_{\ell',n}\left(\omega_{\ell} 
[\omega_{\ell',n} - \omega_{\ell'}] + \omega_{\ell'} [\omega_{\ell,n} - \omega_{\ell}] 
+ [\omega_{\ell,n} - \omega_{\ell}] [\omega_{\ell',n} - \omega_{\ell'}]\right).
\end{split}
\end{equation}

Similar equalities hold after we replace $\mathbf{t}$ by $\mathbf{m}$ and $\mathbf{t}_n$ by $\mathbf{m}_n$. Using the previous computations, we obtain
\begin{equation}
\label{eq:errorF}
\begin{split}
&F_n(\mathbf{t}_n)-F_n(\mathbf{m}_n) - [F_{\beta,h}(\mathbf{t})-F_{\beta,h}(\mathbf{m})]\\
&=F_n(\mathbf{t}_n)-F_{\beta,h}(\mathbf{t}_n) + F_{\beta,h}(\mathbf{t}_n)-F_{\beta,h}(\mathbf{t})  
- [F_n(\mathbf{m}_n)-F_{\beta,h}(\mathbf{m}_n) + F_{\beta,h}(\mathbf{m}_n)-F_{\beta,h}(\mathbf{m}) ]\\
&= -\frac{1}{2} \sum_{\ell,\ell' \in [k]}a_{\ell} \, a_{\ell'}\\
&\times  \left[\mathbf{t}_{\ell,n} \mathbf{t}_{\ell',n} - \mathbf{m}_{\ell,n} \mathbf{m}_{\ell',n}\right]
\left(\omega_{\ell} [\omega_{\ell',n} - \omega_{\ell'}] + \omega_{\ell'} [\omega_{\ell,n} - \omega_{\ell}] 
+ [\omega_{\ell,n} - \omega_{\ell}] [\omega_{\ell',n} - \omega_{\ell'}]\right)\\
& \quad-\frac{1}{2} \sum_{\ell,\ell' \in [k]}a_{\ell} \,a_{\ell'} \,\omega_{\ell}\, \omega_{\ell'} 
\left[\mathbf{t}_{\ell,n}\mathbf{t}_{\ell',n} - \mathbf{t}_{\ell}\mathbf{t}_{\ell'} 
+ \mathbf{m}_{\ell}\mathbf{m}_{\ell'} - \mathbf{m}_{\ell,n}\mathbf{m}_{\ell',n} \right]\\
&\quad -h \sum_{\ell \in [k]} [\omega_{\ell,n} - \omega_{\ell}] \left[\mathbf{t}_{\ell,n}-\mathbf{m}_{\ell,n}\right]\\
&\quad +\frac{1}{\beta} \sum_{\ell \in [k]} [\omega_{\ell,n} - \omega_{\ell}] 
\left[I_{\mathbf{C}}(\mathbf{t}_{\ell,n})-I_{\mathbf{C}}(\mathbf{m}_{\ell,n})\right] 
+\frac{1}{\beta} \sum_{\ell \in [k]} \frac{1}{2n} \log \left(\frac{1-\mathbf{t}_{\ell,n}^2}{1-\mathbf{m}_{\ell,n}^2}\right)\\
&\quad + \frac{1}{\beta} \sum_{\ell \in [k]} \omega_{\ell}  \left[I_{\mathbf{C}}(\mathbf{t}_{\ell,n})-I_{\mathbf{C}}(\mathbf{t}_{\ell}) 
+ I_{\mathbf{C}}(\mathbf{m}_{\ell}) - I_{\mathbf{C}}(\mathbf{m}_{\ell,n})\right] +o\left(n^{-1}\right).
\end{split}
\end{equation}
Using\eqref{eq:Taylor_Ic}, we find
\begin{equation}
\label{eq:errorF_2}
\begin{split}
&[F_n(\mathbf{t}_n)-F_n(\mathbf{m}_n)] - [F_{\beta,h}(\mathbf{t})-F_{\beta,h}(\mathbf{m})]\\
&= -\frac{1}{2} \sum_{\ell,\ell' \in [k]}a_{\ell} \, a_{\ell'} 
\left[\mathbf{t}_{\ell,n} \mathbf{t}_{\ell',n} - \mathbf{m}_{\ell,n} \mathbf{m}_{\ell',n}\right]\\
&\quad \times \left(\omega_{\ell} [\omega_{\ell',n} - \omega_{\ell'}] + \omega_{\ell'} [\omega_{\ell,n} - \omega_{\ell}] 
+ [\omega_{\ell,n} - \omega_{\ell}] [\omega_{\ell',n} - \omega_{\ell'}]\right)\\
&\quad -\frac{1}{2} \sum_{\ell,\ell' \in [k]}a_{\ell} \,a_{\ell'} \,\omega_{\ell}\, \omega_{\ell'} 
\left[\mathbf{t}_{\ell,n}\mathbf{t}_{\ell',n} - \mathbf{t}_{\ell}\mathbf{t}_{\ell'} 
+ \mathbf{m}_{\ell}\mathbf{m}_{\ell'} - \mathbf{m}_{\ell,n}\mathbf{m}_{\ell',n} \right]\\
&\quad -h \sum_{\ell \in [k]} [\omega_{\ell,n} - \omega_{\ell}] \left[\mathbf{t}_{\ell,n}-\mathbf{m}_{\ell,n}\right]\\
&\quad +\frac{1}{\beta} \sum_{\ell \in [k]} [\omega_{\ell,n} - \omega_{\ell}] 
\left[I_{\mathbf{C}}(\mathbf{t}_{\ell,n})-I_{\mathbf{C}}(\mathbf{m}_{\ell,n})\right] 
+\frac{1}{\beta} \sum_{\ell \in [k]} \frac{1}{2n} \log \left(\frac{1-\mathbf{t}_{\ell,n}^2}{1-\mathbf{m}_{\ell,n}^2}\right)\\
&\quad + \frac{1}{\beta} \sum_{\ell \in [k]} \omega_{\ell}  
\Big[(\mathbf{t}_{\ell,n}-\mathbf{t}_{\ell}) \beta \left[ a_{\ell} K(\mathbf{t}) +h\right] + O((\mathbf{t}_{\ell,n}-\mathbf{t}_{\ell})^2)\\
&\qquad  - (\mathbf{m}_{\ell,n}-\mathbf{m}_{\ell}) \beta \left[ a_{\ell} K(\mathbf{m}) +h\Big] 
+ O((\mathbf{m}_{\ell,n}-\mathbf{m}_{\ell})^2)\right] \\
&\quad +o\left(n^{-1}\right).
\end{split}
\end{equation}
Since
\begin{equation}
\label{eq:tt-tt}
\mathbf{t}_{\ell,n}\mathbf{t}_{\ell',n} - \mathbf{t}_{\ell}\mathbf{t}_{\ell'}
= \left(\mathbf{t}_{\ell} [\mathbf{t}_{\ell',n} - \mathbf{t}_{\ell'}] + \mathbf{t}_{\ell'} [\mathbf{t}_{\ell,n} 
- \mathbf{t}_{\ell}] + [\mathbf{t}_{\ell,n} - \mathbf{t}_{\ell}] [\mathbf{t}_{\ell',n} - \mathbf{t}_{\ell'}]\right),
\end{equation}
we focus on estimating $\mathbf{t}_{\ell,n} -\mathbf{t}_{\ell}$.

From Taylor expansion, we get
\begin{equation}
\begin{split}
& \mathbf{t}_{\ell,n} -\mathbf{t}_{\ell} =   \tanh\left(\beta \left[a_{\ell}\sum_{\ell' \in [k]} 
a_{\ell'}\,\omega_{\ell',n}\,\mathbf{t}_{\ell',n}+h\right]\right) 
- \tanh\left(\beta \left[a_{\ell}\sum_{\ell' \in [k]} a_{\ell'}\,\omega_{\ell'}\,\mathbf{t}_{\ell'}+h\right]\right)\\
&=\beta\,a_{\ell}\sum_{\ell' \in [k]}a_{\ell'} [\omega_{\ell',n}\,\mathbf{t}_{\ell',n}-\omega_{\ell'}\,\mathbf{t}_{\ell'}]  
\left[1- \tanh\left(\beta \left[a_{\ell}\sum_{\ell' \in [k]} a_{\ell'}\,\omega_{\ell'}\,\mathbf{t}_{\ell'}+h\right] \right)^2 \right]\\
&\quad - \beta^2\,a_{\ell}^2\left(\sum_{\ell' \in [k]}a_{\ell'} [\omega_{\ell',n}\,\mathbf{t}_{\ell',n}
-\omega_{\ell'}\,\mathbf{t}_{\ell'}] \right)^2  \tanh\left(\beta \left[a_{\ell}\sum_{\ell' \in [k]} 
a_{\ell'}\,\omega_{\ell'}\,\mathbf{t}_{\ell'}+h\right] \right)\\
&\quad \times \left[1- \tanh\left(\beta \left[a_{\ell}\sum_{\ell' \in [k]} 
a_{\ell'}\,\omega_{\ell'}\,\mathbf{t}_{\ell'}+h\right] \right)^2 \right] \\
 &\quad+ O\left(a_{\ell}^3\left(\sum_{\ell' \in [k]}a_{\ell'} [\omega_{\ell',n}\,\mathbf{t}_{\ell',n}
 -\omega_{\ell'}\,\mathbf{t}_{\ell'}] \right)^3 \right).
\end{split}
\end{equation}
Since
\begin{equation}
\begin{split}
\omega_{\ell',n}\,\mathbf{t}_{\ell',n}-\omega_{\ell'}\,\mathbf{t}_{\ell'}
&=(\omega_{\ell',n}\,-\omega_{\ell'})\,\mathbf{t}_{\ell'} +\omega_{\ell',n} (\mathbf{t}_{\ell',n}-\mathbf{t}_{\ell'}),
\end{split}
\end{equation}
we have
\begin{equation}
\begin{split}
\mathbf{t}_{\ell,n} -\mathbf{t}_{\ell}  
&=\beta\,a_{\ell} \left[1- \mathbf{t}_{\ell}^2 \right] \sum_{\ell' \in [k]}a_{\ell'} [(\omega_{\ell',n}\,-\omega_{\ell'})\,
\mathbf{t}_{\ell'} +\omega_{\ell',n} (\mathbf{t}_{\ell',n}-\mathbf{t}_{\ell'})]  \\
&\quad - \beta^2\,a_{\ell}^2  \,\mathbf{t}_{\ell} \left[1- \mathbf{t}_{\ell}^2 \right] 
\left(\sum_{\ell' \in [k]}a_{\ell'} [(\omega_{\ell',n}\,-\omega_{\ell'})\,\mathbf{t}_{\ell'} 
+\omega_{\ell',n} (\mathbf{t}_{\ell',n}-\mathbf{t}_{\ell'})] \right)^2  \\
&\quad+ O\left(a_{\ell}^3\left(\sum_{\ell' \in [k]}a_{\ell'} [(\omega_{\ell',n}\,-\omega_{\ell'})\,\mathbf{t}_{\ell'} 
+\omega_{\ell',n} (\mathbf{t}_{\ell',n}-\mathbf{t}_{\ell'})] \right)^3 \right).
\end{split}
\end{equation}
Suppose that $\mathbf{t}_{\ell,n} -\mathbf{t}_{\ell} \sim \frac{Y_{\ell}^{\mathbf{t}}}{\sqrt{n}}$. By the Central Limit Theorem, $\omega_{\ell,n} -\omega_{\ell} \sim \frac{Z_{\ell}}{\sqrt{n}}$, where $Z_{\ell}$ is the normal random variable  $N(0,\omega_{\ell}(1-\omega_{\ell}))$. Hence
\begin{equation}
\begin{split}
\frac{Y_{\ell}^{\mathbf{t}}}{\sqrt{n}}  
&=\beta\,a_{\ell} \left[1- \mathbf{t}_{\ell}^2 \right] \sum_{\ell' \in [k]}a_{\ell'} \left[\frac{Z_{\ell'}}{\sqrt{n}} \mathbf{t}_{\ell'} 
+\left(\frac{Z_{\ell'}}{\sqrt{n}}+\omega_{\ell'}\right) \frac{Y_{\ell'}^{\mathbf{t}}}{\sqrt{n}}\right]  \\
&\quad - \beta^2\,a_{\ell}^2  \,\mathbf{t}_{\ell} \left[1- \mathbf{t}_{\ell}^2 \right] \left(\sum_{\ell' \in [k]}a_{\ell'} 
\left[\frac{Z_{\ell'}}{\sqrt{n}}\,\mathbf{t}_{\ell'} +\left(\frac{Z_{\ell'}}{\sqrt{n}}+\omega_{\ell'}\right) 
\frac{Y_{\ell'}^{\mathbf{t}}}{\sqrt{n}}\right] \right)^2  \\
&\quad+ O\left(a_{\ell}^3\left(\sum_{\ell' \in [k]}a_{\ell'} \left[\frac{Z_{\ell'}}{\sqrt{n}}\,\mathbf{t}_{\ell'} 
+ \left(\frac{Z_{\ell'}}{\sqrt{n}}+\omega_{\ell'}\right) \frac{Y_{\ell'}^{\mathbf{t}}}{\sqrt{n}}\right] \right)^3 \right)\\
& =\frac{1}{\sqrt{n}}  \beta\,a_{\ell} \left[1- \mathbf{t}_{\ell}^2 \right] \sum_{\ell' \in [k]}a_{\ell'} 
\left(\mathbf{t}_{\ell'}Z_{\ell'}   + \omega_{\ell'} Y_{\ell'}^{\mathbf{t}} \right) \\
&\quad + \frac{1}{n} \beta\,a_{\ell} \left[1- \mathbf{t}_{\ell}^2 \right] \sum_{\ell' \in [k]}a_{\ell'} 
Z_{\ell'} \left(Y_{\ell'}^{\mathbf{t}}  
- \beta \,a_{\ell}\mathbf{t}_{\ell}\mathbf{t}_{\ell'}  \sum_{\ell'' \in [k]} \, a_{\ell''}  \omega_{\ell''} 
Y_{\ell''}^{\mathbf{t}}\right) + o(n^{-1})
\end{split}
\end{equation}
and so
\begin{equation}
\begin{split}
\label{eq:Y(Z)}
Y_{\ell}^{\mathbf{t}}
&= \beta \,a_{\ell} \left[1- \mathbf{t}_{\ell}^2 \right] \frac{\sum_{\ell' \in [k]} 
a_{\ell'} \mathbf{t}_{\ell'} Z_{\ell'}}{1-\beta \sum_{\ell' \in [k]} a_{\ell'}^2 \omega_{\ell'} \left[1- \mathbf{t}_{\ell'}^2 \right] }   
+ O(n ^{-\frac{1}{2}}),
\end{split}
\end{equation}
where the denominator does not vanish because of Remark~\ref{rem:gamma_cond}. Thus, up to a factor  $O(n^{-\frac{1}{2}})$, $Y_{\ell}^{\mathbf{t}}$ is a normal random variable with mean $0$ and variance 
\begin{equation}
\label{eq:varY}
\beta^2 \,a_{\ell}^2 \left[1- \mathbf{t}_{\ell}^2 \right]^2 \frac{	\sum_{\ell' \in [k]} a_{\ell'}^2 \mathbf{t}_{\ell'}^2 
\omega_{\ell'}(1-\omega_{\ell'})}{\left(1-\beta \sum_{\ell' \in [k]} a_{\ell'}^2 \omega_{\ell'} 
\left[1- \mathbf{t}_{\ell'}^2 \right] \right)^2}.
\end{equation}
Similar results hold after we replace $\mathbf{t}$ by $\mathbf{m}$.

Going back to \eqref{eq:errorF_2}, using \eqref{eq:tt-tt} and \eqref{eq:Y(Z)}, and inserting $\mathbf{t}_{\ell,n} -\mathbf{t}_{\ell} \sim \frac{Y_{\ell}^{\mathbf{t}}}{\sqrt{n}}$ and $\mathbf{m}_{\ell,n} -\mathbf{m}_{\ell} \sim \frac{Y_{\ell}^{\mathbf{m}}}{\sqrt{n}}$ and $\omega_{\ell,n} -\omega_{\ell} \sim \frac{Z_{\ell}}{\sqrt{n}}$, we obtain
\begin{equation}
	\begin{split}
		&[F_n(\mathbf{t}_n)-F_n(\mathbf{m}_n)] - [F_{\beta,h}(\mathbf{t})-F_{\beta,h}(\mathbf{m})]\\
		&\sim -\frac{1}{2} \sum_{\ell,\ell' \in [k]} a_{\ell} \, a_{\ell'} \left[\left(\mathbf{t}_{\ell}
		+ \frac{Y_{\ell}^{\mathbf{t}}}{\sqrt{n}}\right) \left(\mathbf{t}_{\ell'}
		+ \frac{Y_{\ell'}^{\mathbf{t}}}{\sqrt{n}}\right) -\left(\mathbf{m}_{\ell} +\frac{Y_{\ell}^{\mathbf{m}}}{\sqrt{n}}\right) 
		\left(\mathbf{m}_{\ell'}+\frac{Y_{\ell'}^{\mathbf{m}}}{\sqrt{n}}\right) \right]\\
		&\qquad \times \left(\omega_{\ell} \frac{Z_{\ell'}}{\sqrt{n}} + \omega_{\ell'} \frac{Z_{\ell}}{\sqrt{n}} + \frac{Z_{\ell}Z_{\ell'}}{n}\right)\\
		& \quad-\frac{1}{2} \sum_{\ell,\ell' \in [k]}a_{\ell} \,a_{\ell'} \,\omega_{\ell}\, \omega_{\ell'} 
		\left(\mathbf{t}_{\ell}\frac{Y_{\ell'}^{\mathbf{t}}}{\sqrt{n}} + \mathbf{t}_{\ell'} \frac{Y_{\ell}^{\mathbf{t}}}{\sqrt{n}} 
		+ \frac{Y_{\ell}^{\mathbf{t}} Y_{\ell'}^{\mathbf{t}}}{n} -\mathbf{m}_{\ell}\frac{Y_{\ell'}^{\mathbf{m}}}{\sqrt{n}} 
		- \mathbf{m}_{\ell'} \frac{Y_{\ell}^{\mathbf{m}}}{\sqrt{n}} - \frac{Y_{\ell}^{\mathbf{m}} Y_{\ell'}^{\mathbf{m}}}{n}\right)\\
		&\quad -h \sum_{\ell \in [k]} \frac{Z_{\ell}}{\sqrt{n}} \left(\mathbf{t}_{\ell}
		+\frac{Y_{\ell}^{\mathbf{t}}}{\sqrt{n}} -\mathbf{m}_{\ell}-\frac{Y_{\ell}^{\mathbf{m}}}{\sqrt{n}}\right)\\
		&\quad +\frac{1}{\beta} \sum_{\ell \in [k]} \frac{Z_{\ell}}{\sqrt{n}} \left[I_{\mathbf{C}}\left(\mathbf{t}_{\ell}
		+\frac{Y_{\ell}^{\mathbf{t}}}{\sqrt{n}}\right)-I_{\mathbf{C}}\left(\mathbf{m}_{\ell}
		+\frac{Y_{\ell}^{\mathbf{m}}}{\sqrt{n}}\right) \right]
		+ \frac{1}{\beta} \sum_{\ell \in [k]} \frac{1}{2n} \log \left(\frac{1-\left(\mathbf{t}_{\ell}
			+\frac{Y_{\ell}^{\mathbf{t}}}{\sqrt{n}}\right)^2}{1-\left(\mathbf{m}_{\ell}
			+\frac{Y_{\ell}^{\mathbf{m}}}{\sqrt{n}}\right)^2}\right)\\
		&\quad + \frac{1}{\beta} \sum_{\ell \in [k]} \omega_{\ell}  \left[\frac{Y_{\ell}^{\mathbf{t}}}{\sqrt{n}}
		\beta \left[ a_{\ell} K(\mathbf{t}) +h\right] + O\left(\frac{(Y_{\ell}^{\mathbf{t}})^2}{n}\right) - \frac{Y_{\ell}^{\mathbf{m}}}{\sqrt{n}}
		\left[ a_{\ell} K(\mathbf{m}) +h\right] + O\left(\frac{(Y_{\ell}^{\mathbf{m}})^2}{n}\right)\right] \\ 
		&\qquad + o\left(n^{-1}\right).\\
	\end{split}
\end{equation}
Thus,
\begin{equation}
	\label{eq:diffF}
	\begin{split}
		&[F_n(\mathbf{t}_n)-F_n(\mathbf{m}_n)] - [F_{\beta,h}(\mathbf{t})-F_{\beta,h}(\mathbf{m})]\\
		&= -\frac{1}{2} \sum_{\ell,\ell' \in [k]}a_{\ell} \, a_{\ell'} \left[\mathbf{t}_{\ell}\mathbf{t}_{\ell'} 
		- \mathbf{m}_{\ell}\mathbf{m}_{\ell'}\right] 
		\left(\omega_{\ell} \frac{Z_{\ell'}}{\sqrt{n}} + \omega_{\ell'} \frac{Z_{\ell}}{\sqrt{n}} \right) \\
		&\quad -\frac{1}{2} \sum_{\ell,\ell' \in [k]}a_{\ell} \,a_{\ell'} \,\omega_{\ell}\, \omega_{\ell'} 
		\left(\mathbf{t}_{\ell}\frac{Y_{\ell'}^{\mathbf{t}}}{\sqrt{n}} 
		+ \mathbf{t}_{\ell'} \frac{Y_{\ell}^{\mathbf{t}}}{\sqrt{n}} -\mathbf{m}_{\ell}\frac{Y_{\ell'}^{\mathbf{m}}}{\sqrt{n}} 
		- \mathbf{m}_{\ell'} \frac{Y_{\ell}^{\mathbf{m}}}{\sqrt{n}}\right)\\
		&\quad -h \sum_{\ell \in [k]}\left[ \mathbf{t}_{\ell} - \mathbf{m}_{\ell} \right]\frac{Z_{\ell}}{\sqrt{n}} 
		+\frac{1}{\beta} \sum_{\ell \in [k]} \frac{Z_{\ell}}{\sqrt{n}} \left[I_{\mathbf{C}}\left(\mathbf{t}_{\ell}
		+\frac{Y_{\ell}^{\mathbf{t}}}{\sqrt{n}}\right) - I_{\mathbf{C}}\left(\mathbf{m}_{\ell}
		+\frac{Y_{\ell}^{\mathbf{m}}}{\sqrt{n}}\right)\right]\\
		&\quad +  \sum_{\ell \in [k]} \omega_{\ell} \left[ \frac{Y_{\ell}^{\mathbf{t}}}{\sqrt{n}} \left[ a_{\ell} K(\mathbf{t}) +h\right] 
		- \frac{Y_{\ell}^{\mathbf{m}}}{\sqrt{n}}
		\left[ a_{\ell} K(\mathbf{m}) +h\right] \right]  +O\left(n^{-1}\right).
	\end{split}
\end{equation}
Since the random variables $Y_{\ell}^{\mathbf{t}}$, $Y_{\ell}^{\mathbf{m}}$, $Z_{\ell}$ are centred normal, this concludes the proof of Theorem~\ref{thm:metER_lim}.

From \eqref{eq:diffF} it is possible to compute explicitly the variance of $Z$ defined in Theorem~\ref{thm:metER_lim}, because the variances of all the random variables involved are known (at least to leading order).


\appendix


\section{Metastability on the complete graph without disorder} 
\label{app:CW} 

We give a brief overview of well-known results for the standard Curie-Weiss model. We refer to \cite[Chapter 13]{BdH15} for more details.

The Glauber dynamics is defined as in Section~\ref{sec:model}, but with $J \equiv 1$. For convenience we write the Curie-Weiss Hamiltonian as
\begin{equation}
\label{eq:CW}
H_{n}(\sigma) = - \frac{1}{2n} \sum_{i,j \in [n]} \sigma(i)\sigma(j)
- h \sum_{i \in [n]} \sigma(i), \qquad \sigma \in \cS_n,
\end{equation}
which is as \eqref{eq:Hnalt} when $J\equiv 1$. What makes this case easier than the one with disorder is that the interaction is \emph{mean-field}. Indeed, we may write
\begin{equation}
H_n(\sigma) = n \big[-\tfrac12 m_n(\sigma)^2-hm_n(\sigma)\big],
\end{equation}
with
\begin{equation}
m_n(\sigma) = \frac{1}{n} \sum_{i \in [n]} \sigma(i) \in [-1,1]
\end{equation}
the magnetisation. In this case the magnetisation process $(m_n(t))_{t \geq 0}$, defined by 
\begin{equation}
\label{eq:maggr}
m_n(t) = m_n(\sigma_t),
\end{equation}
is Markovian. More specifically, it is a nearest-neighbour random walk on the grid 
\begin{equation}
\label{eq:grid}
\Gamma_n = \left\{-1,-1+\tfrac{2}{n},\ldots,+1-\tfrac{2}{n},+1\right\}.
\end{equation}
In the limit as $n\to\infty$, \eqref{eq:maggr} converges to a Brownian motion on $[-1,+1]$ in the potential $F_{\beta,h}$ given by 
\begin{equation}
\label{eq:feCW}
F_{\beta,h}(m) = - \frac12 m^2 - h m + \frac{1}{\beta} I(m),
\end{equation}
with
\begin{equation}
\label{eq:Idef}
I(m) = \frac{1-m}{2} \log \left(\frac{1-m}{2}\right) + \frac{1+m}{2} \log \left(\frac{1+m}{2}\right)
\end{equation} 
the relative entropy of the Bernoulli measure on $\{-1,+1\}$ with parameter $m$ with respect to the counting measure on $\{-1,+1\}$. $F_{\beta,h}(m)$ is the \emph{free energy} at magnetisation $m$, consisting of an \emph{energy term} $- \frac12 m^2 - h m$ and an \emph{entropy term} $\frac{1}{\beta}I(m)$.  See \cite[Chapter 13]{BdH15}
for more details.

Since
\begin{equation}
F'_{\beta,h}(m) = -m-h+\frac{1}{2\beta} \log \left(\frac{1+m}{1-m}\right), \qquad
F''_{\beta,h}(m) = -1 -\frac{1}{\beta}\frac{m}{1-m^2},
\end{equation}
the stationary points of $F_{\beta,h}$ are the solutions to the equation
\begin{equation}
\label{eq:Tmf}
m = T_{\beta,h}(m), \qquad T_{\beta,h}(m) = \tanh[\beta(m+h)].
\end{equation}
Since 
\begin{equation}
\label{eq:Tprime}
T'_{\beta,h}(m) = \beta\big[1-T^2_{\beta,h}(m)\big],
\end{equation}
$T_{\beta,h}$ is strictly increasing and has a unique inflection point at $m=-h$. Consequently, \eqref{eq:Tmf} has either one or three solutions. The latter occurs if and only if
\begin{equation}
\label{eq:metreg}
\beta \in (\bar{\beta}_c, \infty) \qquad \text{and } \qquad h \in (0,h_c(\beta)),
\end{equation}
where $\bar{\beta}_c = 1$ is the \emph{critical inverse temperature} and $\bar{h}_c(\beta)$ is the \emph{critical magnetic field}, i.e., the unique value of $h$ for which $T_{\beta,h}$ touches the diagonal at a unique value of the magnetisation, say $-m(\beta)$. Clearly, $1=\beta(1-m^2(\beta))$, i.e., 
\begin{equation}
m(\beta)=\sqrt{1-\beta^{-1}},
\end{equation} 
and so $\bar{h}_c(\beta)$ solves the equation $T_{\beta,\bar{h}_c(\beta)}(-m(\beta))=-m(\beta)$. Hence (see Fig.~\ref{fig:threshold})
\begin{equation}
\label{eq:hc}
\bar{h}_c(\beta) = m(\beta) - \frac{1}{2\beta} \log \left(\frac{1+m(\beta)}{1-m(\beta)}\right), \qquad \beta \geq 1.
\end{equation}

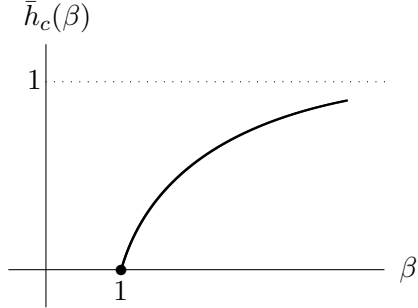
\begin{figure}[htbp]
\begin{center}
\setlength{\unitlength}{0.5cm}
\begin{picture}(10,6)(0,1)
\put(-1,0){\line(10,0){10}}
\put(0,-1){\line(0,7){7}}
\qbezier[40](0,5)(4.5,5)(9,5)
{\thicklines
\qbezier(2,0)(3,3.5)(8,4.5)
}
\put(-.5,4.8){$1$}
\put(1.8,-.8){$1$}
\put(-.6,6.5){$\bar{h}_c(\beta)$}
\put(9.4,-.2){$\beta$}
\put(2,0){\circle*{.3}}
\end{picture}
\end{center}
\vspace{1cm}
\caption{Plot of $\beta \mapsto \bar{h}_c(\beta)$.}
\label{fig:threshold}
\end{figure}

The range of parameters in \eqref{eq:metreg} represents the \emph{metastable regime} in which $F_{\beta,h}$ has a \emph{double-well} shape and, in the limit as $n\to\infty$, the Gibbs measure $\mu_n$ in \eqref{eq:mu} has two phases given by the two minima of $F_{\beta,h}$: the \emph{metastable phase} with magnetisation $\mathbf{m}<0$ and the \emph{stable phase} with magnetisation $\mathbf{s}>0$. The unique \emph{saddle point} in the gate $\mathcal{G}(\mathbf{m},\mathbf{s})$ has magnetisation $\mathbf{t}<0$  (see Fig.~\ref{fig:CWfe}).

\begin{figure}[htbp]
\begin{center}
\setlength{\unitlength}{0.5cm}
\begin{picture}(10,6)(-4.5,-.4)
\put(-6,0){\line(12,0){12}}
\put(0,-3){\line(0,7){7}}
\qbezier[30](2,0)(2,-1)(2,-2)
\qbezier[30](-3,0)(-3,-.5)(-3,-1)
\qbezier[20](-1,0)(-1,.5)(-1,.75)
\qbezier[50](4,0)(4,2)(4,5)
\qbezier[50](-5,0)(-5,2)(-5,5)
\qbezier[80](1,-1.3)(0,0)(-1,1.3)
{\thicklines
\qbezier(.5,-1.2)(.7,-2)(2,-2)
\qbezier(.5,-1.2)(.2,-.4)(0,0)
\qbezier(2,-2)(3.5,-2)(4,3)
\qbezier(-3,-1)(-2.5,-1)(-2,0)
\qbezier(-3,-1)(-4.5,-1)(-5,4)
\qbezier(0,0)(-1,1.5)(-2,0)
}
\put(6.5,-.1){$m$}
\put(-.8,4.5){$F_{\beta,h}(m)$}
\put(-3.3,.5){$\mathbf{m}$}
\put(1.8,.5){$\mathbf{s}$}
\put(-1.15,-.7){$\mathbf{t}$}
\put(-1.7,1.5){$-h$}
\put(3.9,-.7){$1$}
\put(-5.3,-.7){$-1$}
\put(-1,.75){\circle*{.3}}
\put(2,-2){\circle*{.3}}
\put(-3,-1){\circle*{.3}}
\end{picture}
\end{center}
\vspace{1cm}
\caption{Plot of $m \mapsto F_{\beta,h}(m)$ for $\beta,h$ in the metastable regime.}
\label{fig:CWfe}
\end{figure}
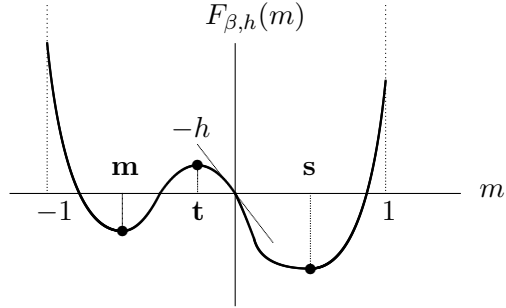

Theorems~\ref{thm: class CW result}--\ref{thm: class exp_law} can be found in Bovier and den Hollander~\cite[Chapter 13]{BdH15}. Here the notation is the same as the one in Section\ref{sec:intro}.  Let $\cS_n[\mathbf{m}]$, $\cS_n[\mathbf{s}]$ denote the sets of configurations in $\cS_n$ for which the magnetisation is closest to $\mathbf{m}$, $\mathbf{s}$, respectively. 

\begin{theorem}[{\bf Average crossover time}]
\label{thm: class CW result}
$\mbox{}$\\
Subject to \eqref{eq:metreg}, uniformly in $\sigma \in \cS_n[\mathbf{m}]$,
\begin{equation}
\mathbb{E}_\sigma\left[\tau_{\cS_n[\mathbf{s}]}\right]
= [1+o_n(1)]\,\frac{\pi}{1-\mathbf{t}} \sqrt{\frac{1-\mathbf{t}^{2}}{1-\mathbf{m}^{2}}}
\frac{1}{\beta\sqrt{F_{\beta,h}''(\mathbf{m})[-F_{\beta,h}''(\mathbf{t})]}}
\,\ee^{\beta n [F_{\beta,h}(\mathbf{t})-F_{\beta,h}(\mathbf{m})]}.
\end{equation}
\end{theorem}

\begin{theorem}[{\bf Exponential law}] 
\label{thm: class exp_law}
$\mbox{}$\\
Subject to \eqref{eq:metreg}, uniformly in $\sigma \in \cS_n[\mathbf{m}]$,
\begin{equation}
\mathbb{P}_{\si} \left( \tau_{\cS_n[\mathbf{s}]}> t \, 
\mathbb{E}_\sigma\left[\tau_{\cS_n[\mathbf{s}]}\right] \right)
= [1+o_n(1)]\,\ee^{-t}, \qquad t \geq 0.
\end{equation}
\end{theorem}

\noindent
Fig.~\ref{fig:CWfe} illustrates the setting: the average crossover time from $\cS_n[\mathbf{m}]$ to $\cS_n[\mathbf{s}]$ depends on the energy barrier $F_{\beta,h}(\mathbf{t})-F_{\beta,h}(\mathbf{m})$ and on the curvature of $F_{\beta,h}$ at $\mathbf{m}$ and $\mathbf{t}$. The crossover time is exponential on the scale of its average.


\section{Examples with multiple metastable states}
\label{app:critical}

We provide examples of distributions and parameter choices (in the metastable regime) for which the model with disorder has multiple critical points. More specifically, we provide numerical evidence that, for $k \in \{2,3,4\}$, \eqref{eq:Keq} can have any number of solutions in the set $\{3,5\ldots,2k+1\}$. The cases with strictly more than 3 solutions present multiple minimal critical points, i.e. multiple metastable states.

\subsection{Case k=2}

\begin{figure}[htbp]
\centering
\subfloat[][\emph{3 critical points}.]
{\includegraphics[width=.45\textwidth]{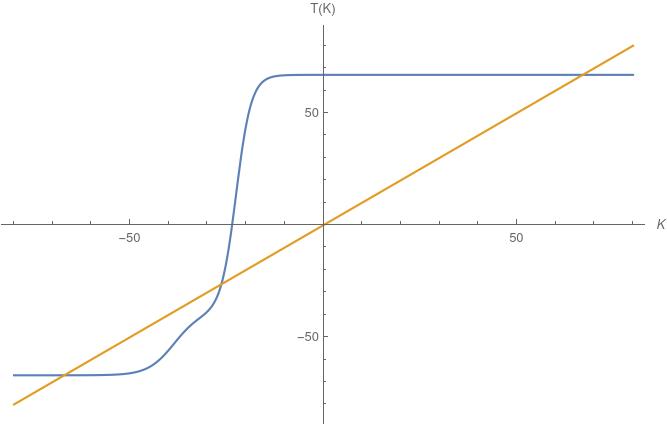} \label{Fig:3k2}} \quad
\subfloat[][\emph{5 critical points}.]
{\includegraphics[width=.45\textwidth]{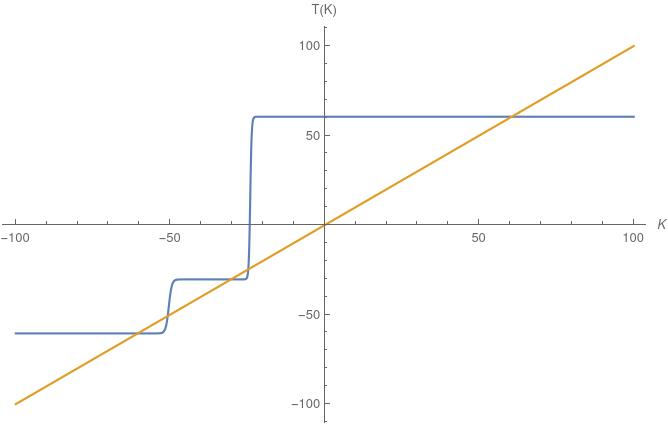} \label{Fig:5k2}}
\caption{$T_{\beta,h}$, $k=2$.}
\label{Fig:k2}
\end{figure}

\begin{itemize}
\item 
Figure~\ref{Fig:3k2}: 3 critical points, parameters $a_1 = 77$, $a_2 = 45$, $\omega_1 = 0.688$,
$h = 1740$, $\beta = 113 \,\beta_c$.
\item 
Figure~\ref{Fig:5k2}: 5 critical points, parameters $a_1 = 774$, $a_2 = 36.84$, $\omega_1 = 0.59$,
$h = 1740$, $\beta = 131 \,\beta_c$.
\end{itemize}


\subsection{Case k=3}

\begin{figure}
\centering
\subfloat[][\emph{3 critical points}.]
{\includegraphics[width=.45\textwidth]{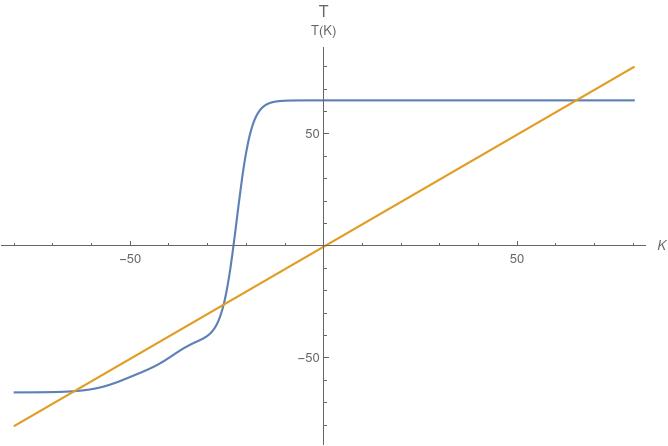} \label{Fig:3k3}} \quad
\subfloat[][\emph{5 critical points}.]
{\includegraphics[width=.45\textwidth]{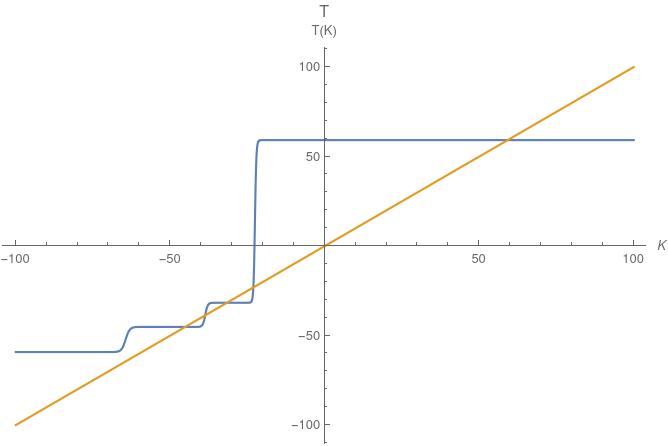} \label{Fig:5k3}} \\
\subfloat[][\emph{7 critical points}.]
{\includegraphics[width=.45\textwidth]{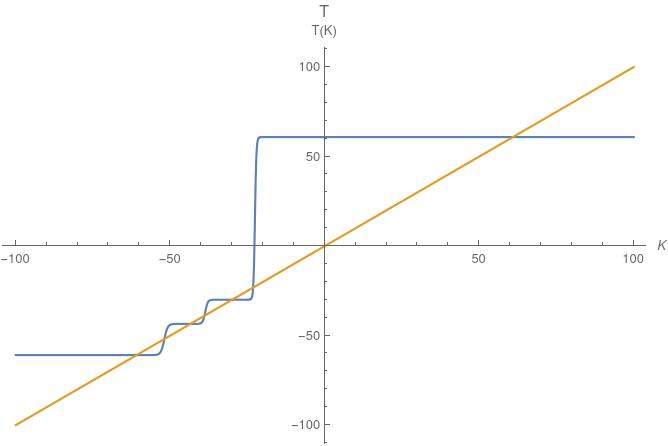} \label{Fig:7k3}} 
\caption{$T_{\beta,h}$, $k=3$.}
\label{Fig:k3}
\end{figure}

\begin{itemize}
\item 
Figure~\ref{Fig:3k3}: 3 critical points, parameters
$a_1 = 77$, $a_2 = 45$, $a_3 = 33.5$, $\omega_1 = 0.688$, $\omega_2 = 0.15$, 
$h = 1740$, $\beta = 113\,\beta_c$.
\item 
Figure~\ref{Fig:5k3}: 5 critical points, parameters
$a_1 = 77$, $a_2 = 45$, $a_3 = 27$, $\omega_1 = 0.59$, $\omega_2 = 0.15$, 
$h = 1740$, $\beta = 113\,\beta_c$.
\item  
Figure~\ref{Fig:7k3}: 7 critical points, parameters
$a_1 = 77$, $a_2 = 45$, $a_3 = 33.5$, $\omega_1 = 0.59$, $\omega_2 = 0.15$,
$h = 1740$, $\beta = 113\,\beta_c$.
\end{itemize}


\subsection{Case k=4}

\begin{figure}
\centering
\subfloat[][\emph{3 critical points}.]
{\includegraphics[width=.45\textwidth]{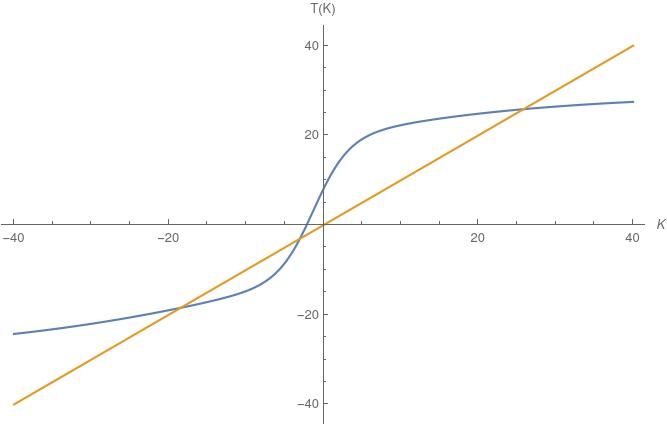} \label{Fig:3k4}} \quad
\subfloat[][\emph{5 critical points}.]
{\includegraphics[width=.45\textwidth]{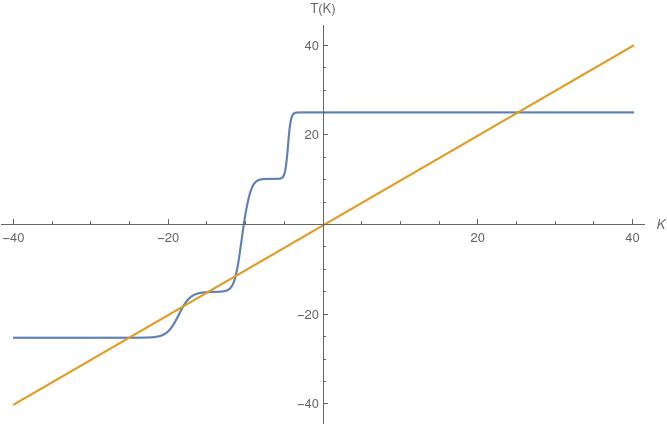} \label{Fig:5k4}} \\
\subfloat[][\emph{7 critical points}.]
{\includegraphics[width=.45\textwidth]{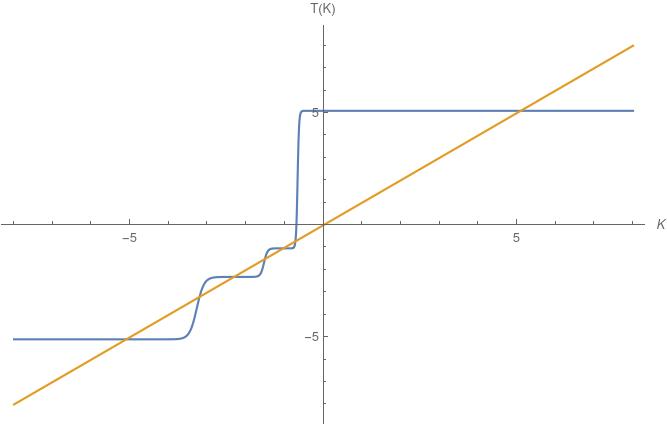} \label{Fig:7k4}} \quad
\subfloat[][\emph{9 critical points}.]
{\includegraphics[width=.45\textwidth]{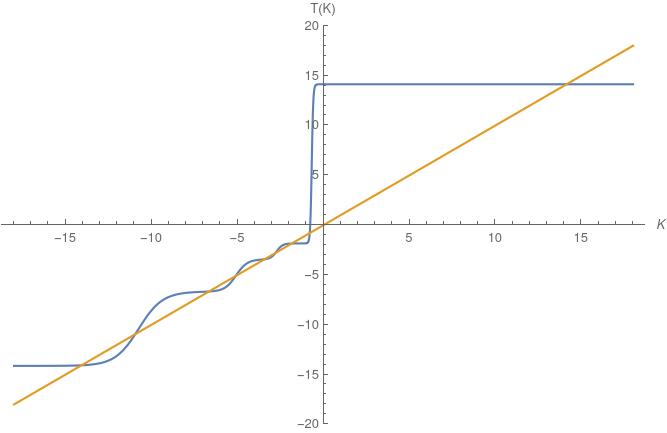} \label{Fig:9k4}}
\caption{$T_{\beta,h}$, $k=4$.}
\label{Fig:k4}
\end{figure}

\begin{itemize}
\item 
Figure~\ref{Fig:3k4}: 3 critical points, parameters
$a_1 = 12$, $a_2 = 16$, $a_3 = 139.5$, $a_4 = 24.5$, $\omega_1 = 0.474$, $\omega_2 = 0.22$, $\omega_3 = 0.111$, 
$h = 178$, $\beta = 3.8\,\beta_c$.
\item 
Figure~\ref{Fig:5k4}: 5 critical points, parameters
$a_1 = 14$,  $a_2 = 27$,	$a_3 = 57$, $a_4 = 24.5$, $\omega_1 = 0.366$, $\omega_2 = 0.1$, $\omega_3 = 0.13$, 
$h = 262$, $\beta = 38.4\,\beta_c$.
\item 
Figure~\ref{Fig:7k4}: 7 critical points, parameters 
$a_1 = 2.32$, $a_2 = 4.92$, $a_3 = 5$, $a_4 = 11.32$, $\omega_1 = 0.6$, $\omega_2 = 0.096$, $\omega_3 = 0.033$, 
$h = 7.6$, $\beta = 95.2 \,\beta_c$.
\item 
Figure~\ref{Fig:9k4}: 9 critical points, parameters
$a_1 = 12$, $a_2 = 16$, $a_3 = 50.5$, $a_4 = 24.5$, $\omega_1 =0 .474$, $\omega_2 = 0.22$, $\omega_3 = 0.111$, 
$h = 178$, $\beta = 63.2 \,\beta_c$.
\end{itemize}


\section{Example of $h_c(\beta)$ not increasing}
\label{app:hc_decr}
We provide here an example of choice of the law of $J$ for which the critical threshold $\beta \mapsto h_c(\beta)$ is not monotone increasing. This implies the possibility of a re-entrant metastable crossover.
 
For $k=4$, pick $a_1 = 12$, $a_2 = 16$, $a_3 = 50.5$, $a_4 = 24.5$ and $\omega_1 = 0.474$, $\omega_2 = 0.22$, $\omega_3 = 0.111$.  Take $h = 100$, and plot the function $K \mapsto T_{\beta,h}(K)$ varying $\beta$. For $\beta_1 = 4 \,\beta_c = 0.00762336$ the system is metastable: $T_{\beta,h}$ intersects the diagonal three times (see Figure~\ref{Fig:betaSmall}), which implies that $h< h_c(\beta_1)$. For $\beta_2 = 21 \,\beta_c =0.04002264 >\beta_1$ the system is not metastable: $T_{\beta,h}$ intersects the diagonal only once (see Figure~\ref{Fig:betaBig}), which implies that $h> h_c(\beta_2)$. This shows that $h_c(\beta)$ is not necessarily an increasing function of $\beta$.

\begin{figure}[htbp]
\centering
\subfloat[][\emph{$h = 100$, $\beta = 0.00762336$}.]
{\includegraphics[width=.45\textwidth]{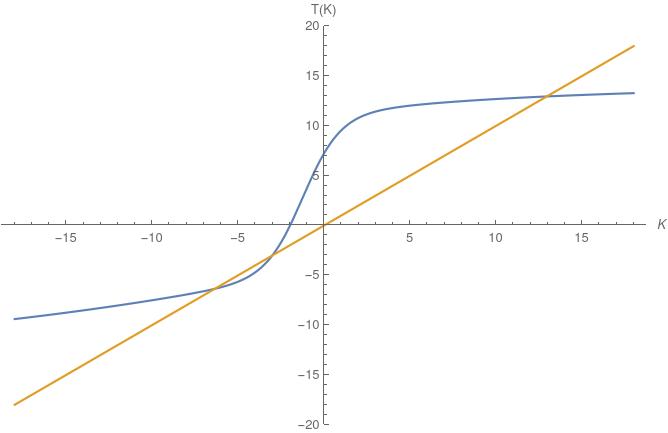} \label{Fig:betaSmall}} \quad
\subfloat[][\emph{$h = 100$, $\beta = 0.04002264$}.]
{\includegraphics[width=.45\textwidth]{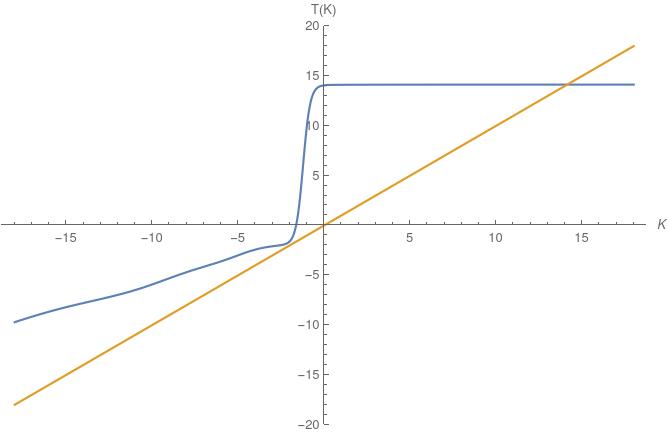} \label{Fig:betaBig}}
\caption{$T_{\beta,h}$, fixed $h$ and law of the components of $J$, varying $\beta$.}
\label{Fig:vary_beta}
\end{figure}	

\section{Limit of the prefactor}
\label{app:prefactor}
Below Theorem~\ref{thm:exp_law} we stated that the prefactor in \eqref{CWERasymp} converges. For completeness, in this Appendix we compute its limit, although, as we mentioned after Theorem~\ref{thm:metER_lim}, it is negligible because of the order of approximation of the exponent.

We focus first on $\gamma_n$. Recall notation in \eqref{eq:Kn}, \eqref{eq:KFirst} and \eqref{eq:def_omegaln}. Then \eqref{eq:gamma} can be written as	
\begin{equation}
\label{eq:limit_eq_gamma}
\begin{split}
&1 +O(n^{-1}) = \sum_{\ell \in [k]} 
\frac{a_{\ell}^2 \oln (1-\mathbf{t}_{\ell,n})
\exp\left[-2\beta\left(-a_{\ell}\left(\frac{a_{\ell}}{n}+K_n(\mathbf{t}_n)\right)-h\right)_+\right]} 
{\frac
{\exp\left[-2\beta\left(-a_{\ell}\left(\frac{a_{\ell}}{n}+K_n(\mathbf{t}_n)\right)-h\right)_+\right]}
{\beta (1+\mathbf{t}_{\ell,n})			} 
-2\gamma_n}\\
&= \sum_{\ell \in [k]} 
\frac{
a_{\ell}^2 \oln \left(1-\tanh\left(\beta\left[a_{\ell}\, K_n(\mathbf{t}_n)+h\right]\right)\right) 
\exp\left[-2\beta\left(-a_{\ell}\left(\frac{a_{\ell}}{n} + K_n(\mathbf{t}_n)\right)-h\right)_+\right] } 
{\frac
{\exp\left[-2\beta\left(-a_{\ell}\left(\frac{a_{\ell}}{n} + K_n(\mathbf{t}_n)\right)-h\right)_+\right]}
{\beta\left(1+\tanh\left(\beta\left[a_{\ell}K_n(\mathbf{t}_n)+h\right]\right)\right)} 
-2\gamma_n}.
\end{split}
\end{equation}
In the first equality we use \eqref{eq:crit} for $\mathbf{t}_n$, i.e., the approximation of the stationary points of $F_n$ by the stationary points of $\bar{F}_n$. This makes $\mathbf{t}_{\ell,n}$ independent of $\ell$, so that we can use the law of large numbers in the limit as $n \to \infty$. Thus, we obtain that $\gamma_n$ converges to $\gamma$, the solution of the equation
\begin{equation}
\label{eq:gamma_limit}
\mathfrak{E}\left(\frac{J(1)^2(1+\tanh U)\,\ee^{-2U_+}}
{\frac{1}{\beta(1-\tanh U)}\,\ee^{-2U_+}-2\gamma}\right) = 1,
\end{equation}
where $\mathfrak{E}$ denotes expectation with respect to $\cP$ and $U = -\beta[J(1)\,K(\mathbf{t})+h]$, with $\mathbf{t}$ solving \eqref{eq:m*}. Note that \eqref{eq:gamma_limit} is similar to \cite[Eq. (14.4.14)]{BdH15}.

We are left to find the limit of the determinants ratio. By \eqref{eq:detAm},
\begin{equation}
\begin{split}
\det \mathbb{A}_n(m)
&= \left(1-\sum_{\ell \in [k]} \beta \, a_{\ell}^2\, \omega_{\ell,n} [1-(m_{\ell})^2]  \right) 
\prod_{\ell'\in [k]}  \frac{1}{\beta} \frac{\omega_{\ell',n}}{1-(m_{\ell'})^2}\left[1+O(n^{-1})\right].
\end{split}
\end{equation}

Using \eqref{eq:crit} for $m \in \{\mathbf{t}_n, \mathbf{m}_n\}$, we have 
\begin{equation}
\label{eq:limit_n}
\begin{split}
&\sum_{\ell \in [k]} \beta \, a_{\ell}^2\, \omega_{\ell,n} [1-(m_{\ell,n})^2] \\
&=\sum_{\ell \in [k]} \beta \, a_{\ell}^2\, \omega_{\ell,n} \left[1-\tanh^2\left(\beta \left[a_{\ell} 
\sum_{\ell' \in [k]} a_{\ell'}\,\omega_{\ell',n}\,m_{\ell',n}+h\right]\right)\right].
\end{split}
\end{equation}
Using the law of large numbers as above and with the same notation, we find
\begin{equation}
\label{eq:ratio_det}
\begin{split}
&\lim_{n\to\infty}\frac{[-\det (\mathbb{A}_n(\mathbf{t}_n))]}{\det (\mathbb{A}_n(\mathbf{m}_n))}
=  \frac{-1+\mathfrak{E} \left( \beta \, J(1)^2\, \left[1-\tanh^2\left[U(\mathbf{t}) \right]\right] \right)  } 
{1- \mathfrak{E} \left( \beta \, J(1)^2\, \left[1-\tanh^2\left[U(\mathbf{m}) \right]\right] \right) }  
\prod_{\ell'\in [k]} \frac{1-(\mathbf{m}_{\ell'})^2}{1-(\mathbf{t}_{\ell'})^2},
\end{split}
\end{equation}
where $U(\mathbf{x})=-\beta(J(1) K(\mathbf{x})+ h)$.



\begin{thebibliography}{99}
	
	\bibitem{BBI09}
	A.~Bianchi, A.~Bovier, and D.~Ioffe.
	\newblock Sharp asymptotics for metastability in the random field
	{C}urie-{W}eiss model.
	\newblock {\em Electronic Journal of Probability}, 14:paper no.\ 53, 1541--1603, 2009. \url{https://doi.org/10.1214/EJP.v14-673}
	
	\bibitem{BBI12}
	A.~Bianchi, A.~Bovier, and D.~Ioffe.
	\newblock Pointwise estimates and exponential laws in metastable systems via coupling methods.
	\newblock {\em Annals of Probability}, 40(1):339--371, 2012. \url{https://doi.org/10.1214/10-AOP622}
	
	\bibitem{BEGK1}
	A.~Bovier, M.~Eckhoff, V.~Gayrard, and M.~Klein.
	\newblock Metastability in stochastic dynamics of disordered mean-field models.
	\newblock {\em Probability Theory and Related Fields}, 119(1):99--161, 2001. \url{https://doi.org/10.1007/PL00012740}
	
	\bibitem{BEGK2}
	A.~Bovier, M.~Eckhoff, V.~Gayrard, and M.~Klein.
	\newblock Metastability and low lying spectra in reversible {M}arkov chains.
	\newblock {\em Communications In Mathematical Physics}, 228(2):219--255, 2002. \url{https://doi.org/10.1007/s002200200609}
	
	\bibitem{BEGK3}
	A.~Bovier, M.~Eckhoff, V.~Gayrard, and M.~Klein.
	\newblock Metastability in reversible diffusion processes. {I}. {S}harp
	asymptotics for capacities and exit times.
	\newblock {\em Journal of the European Mathematical Society}, 6(4):399--424, 2004. \url{https://doi.org/10.4171/JEMS/14}
	
	\bibitem{BdH15} 
	A.\ Bovier and F.\ den Hollander.
	\emph{Metastability -- A Potential-Theoretic Approach}. 
	Grundlehren der mathematischen Wissenschaften 351, Springer, Cham, 2015.
	\url{https://doi.org/10.1007/978-3-319-24777-9}
	
	\bibitem{BMP21}
	A.~Bovier, S.~Marello, and E.~Pulvirenti.
	\newblock Metastability for the dilute {C}urie-{W}eiss model with {G}lauber dynamics.
	\newblock {\em  Electronic Journal of Probability}, 26:1--38, 2021.
	\url{https://doi.org/10.1214/21-ejp610}
	
	\bibitem{D17}
	S.\ Dommers.
	Metastability of the Ising model on random regular graphs at zero temperature.
	\newblock{Probability Theory and Related Fields} 167:305--324, 2017.
	\url{https://doi.org/10.1007/s00440-015-0682-0}
	
	\bibitem{DGGHP16}
	S.~Dommers, C.~Giardin\`a, C.~Giberti, R.~van~der Hofstad, and M.~L.
	Prioriello.
	\newblock Ising critical behavior of inhomogeneous {C}urie-{W}eiss models and
	annealed random graphs.
	\newblock {\em Communications in Mathematical Physics}, 348(1), 221--263, 2016.
	\url{https://doi.org/10.1007/s00220-016-2752-2}
	
	\bibitem{DdHJN17}
	S.\ Dommers, F.\ den Hollander, O.\ Jovanovski and F.R.\ Nardi. 
	Metastability for Glauber dynamics on random graphs.
	\newblock{Annals of Applied Probability} 27:2130--2158, 2017.
	\url{https://doi.org/10.1214/16-AAP1251}
	
	\bibitem{dHO20}
	F.~den Hollander and O.~Jovanovski.
	\newblock Glauber dynamics on the {E}rd\H{o}s-{R}\'{e}nyi random graph.
	\newblock In {\em In and out of equilibrium. 3. {C}elebrating {V}ladas
		{S}idoravicius}, volume~77 of {\em Progress in Probability}, pages 519--589,
	Birkh\"{a}user/Springer, Cham, 2021.
	\url{https://doi.org/10.1007/978-3-030-60754-8\_24}
	
	\bibitem{MP98}
	P.~Mathieu and P.~Picco.
	\newblock Metastability and convergence to equilibrium for the random field {C}urie-{W}eiss model.
	\newblock {\em Journal of Statistical Physics}, 91(3-4):679--732, 1998.
	\url{https://doi.org/10.1023/A:1023085829152}
	
	\bibitem{OV05}
	E.\ Olivieri and M.E.\ Vares.
	\emph{Large Deviations and Metastability}.
	Encyclopedia of Mathematics and its Applications 100,
	Cambridge University Press, Cambridge, 2005.
	\url{https://doi.org/10.1017/CBO9780511543272}
	
	\bibitem{TC74}
	P.A.J.~Tindemans and H.W.~Capel.
	\newblock An exact calculation of the free energy in systems with separable
	interactions.
	\newblock {\em Physica}, 72(3):433--464, 1974.
	\url{https://doi.org/10.1016/0031-8914(74)90209-2}
	
\end{thebibliography}
\end{document}